\documentclass[12pt]{amsart}
\usepackage{amssymb,amsmath}
\usepackage{graphicx}
\usepackage{enumitem}
\usepackage{accents}

\DeclareSymbolFont{matha}{OML}{txmi}{m}{it}
\DeclareMathSymbol{\varu}{\mathord}{matha}{117}
\DeclareMathSymbol{\varv}{\mathord}{matha}{118}
\DeclareMathSymbol{\varw}{\mathord}{matha}{119}

\newcommand{\ubar}[1]{\underaccent{\bar}{#1}}

\makeatletter
\@namedef{subjclassname@2020}{%
  \textup{2020} Mathematics Subject Classification}
\makeatother

\newcommand{\bA}{\mathbb{A}}
\newcommand{\bC}{\mathbb{C}}

\newcommand{\bP}{\mathbb{P}}
\newcommand{\bQ}{\mathbb{Q}}
\newcommand{\bR}{\mathbb{R}}
\newcommand{\bZ}{\mathbb{Z}}

\newcommand{\cB}{{\mathcal{B}}}
\newcommand{\cC}{{\mathcal{C}}}
\newcommand{\cCv}{\cC_\pv}
\newcommand{\cCw}{\cC_\pw}
\newcommand{\cCxi}{\cC_\uxi}
\newcommand{\cCXi}{\cC_\Xi}

\newcommand{\cK}{{\mathcal{K}}}

\newcommand{\cO}{{\mathcal{O}}}
\newcommand{\cOv}{\cO_\pv}
\newcommand{\cS}{{\mathcal{S}}}

\newcommand{\Dscal}{D}
\newcommand{\Dwedge}{D^*}
\newcommand{\et}{\quad\mbox{and}\quad}

\newcommand{\iprod}{\mathbin{\lrcorner}}
\newcommand{\Kv}{L}
\newcommand{\lambdahat}{\widehat{\lambda}}
\newcommand{\mesh}{c}
\newcommand{\MK}{M(K)}
\newcommand{\MKinf}{M_\infty(K)}
\newcommand{\module}[1]{|#1|_\bA}
\newcommand{\mua}{\module{\ua}}
\newcommand{\mult}{*}
\newcommand{\norm}[1]{\|\hspace*{2pt}#1\hspace*{1pt}\|}
\newcommand{\omegahat}{\widehat{\omega}}
\newcommand{\phibot}{{\ubar{\varphi}}}
\newcommand{\phitop}{{\bar{\varphi}}}
\newcommand{\psibot}{{\ubar{\psi}}}
\newcommand{\psitop}{{\bar{\psi}}}
\newcommand{\ptop}{\perp_{\mathrm{top}}}
\newcommand{\pu}{\varu}
\newcommand{\pv}{\varv}
\newcommand{\pw}{\varw}

\newcommand{\Sw}{S\setminus\{\pw\}}
\newcommand{\tbigwedge}{{\textstyle{\bigwedge}}}

\newcommand{\ua}{\boldsymbol{a}}
\newcommand{\ub}{\boldsymbol{b}}
\newcommand{\ue}{\mathbf{e}}
\newcommand{\uE}{\mathbf{E}}
\newcommand{\ui}{\mathbf{i}}
\newcommand{\uj}{\mathbf{j}}
\newcommand{\uL}{\mathbf{L}}
\newcommand{\uP}{\mathbf{P}}
\newcommand{\uR}{\mathbf{R}}
\newcommand{\uu}{\mathbf{u}}

\newcommand{\uv}{\mathbf{v}}
\newcommand{\uw}{\mathbf{w}}
\newcommand{\ux}{\mathbf{x}}
\newcommand{\uX}{\mathbf{X}}
\newcommand{\uy}{\mathbf{y}}
\newcommand{\uybar}{\overline{\uy}}
\newcommand{\uY}{\mathbf{Y}}
\newcommand{\uz}{\mathbf{z}}
\newcommand{\ualpha}{{\boldsymbol{\alpha}}}

\newcommand{\uomega}{{\boldsymbol{\omega}}}
\newcommand{\uxi}{{\boldsymbol{\xi}}}
\newcommand{\ybar}{\overline{y}}
\newcommand{\wx}{\widehat{\ux}}

\DeclareMathOperator{\dist}{dist}

\newtheorem{theorem}{Theorem}[section]
\newtheorem*{thmA}{Theorem A}
\newtheorem*{thmB}{Theorem B}
\newtheorem*{thmC}{Theorem C}
\newtheorem{lemma}[theorem]{Lemma}
\newtheorem{proposition}[theorem]{Proposition}
\newtheorem{cor}[theorem]{Corollary}

\theoremstyle{remark}
\newtheorem{definition}[theorem]{Definition}

\numberwithin{equation}{section}

\begin{document}

\title[Parametric geometry of numbers and extension of scalars]
{Parametric geometry of numbers over a number field and extension of scalars}
\author{Anthony Po\"els}
\address{
Bureau 116, B\^atiment Braconnier\\
Universit\'e Claude Bernard Lyon 1\\
43, boulevard du 11 novembre 1918\\
69622 Villeurbanne cedex, France}
\email{poels@math.univ-lyon1.fr}
\author{Damien Roy}
\address{
   D\'epartement de Math\'ematiques\\
   Universit\'e d'Ottawa\\
   150 Louis Pasteur\\
   Ottawa, Ontario K1N 6N5, Canada}
\email{droy@uottawa.ca}
\thanks{Work of both authors partially supported by NSERC}
\dedicatory{Dedicated to Jeff Thunder on his 60-th birthday.}

\maketitle

\baselineskip=15pt

\subsection*{Abstract}
The parametric geometry of numbers of Schmidt and Summerer deals with rational
approximation to points in $\bR^n$.  We extend this theory to a
number field $K$ and its completion $K_\pw$ at a place $\pw$ in order to treat
approximation over $K$ to points in $K_\pw^n$.  As a consequence, we find that
exponents of approximation over $\bQ$ in $\bR^n$ have the same spectrum 
as their generalizations over $K$ in $K_\pw^n$.  When $\pw$ has relative degree
one over a place $\ell$ of $\bQ$, we further relate approximation over $K$ to a point
$\uxi$ in $K_\pw^n$, to approximation over $\bQ$ to a point $\Xi$ in 
$\bQ_\ell^{nd}$ obtained from $\uxi$ by extension of scalars, where $d$ is the degree of $K$
over $\bQ$.  By combination with a result of P.~Bel, this allows us to construct 
algebraic curves in $\bR^{3d}$ defined over $\bQ$, of degree $2d$, containing 
points that are very singular with respect to rational approximation.

\subsection*{Keywords} 
Diophantine approximation, simultaneous approximation, exponents
of approximation, transference inequalities, extremal numbers, rational points, 
singular points, spectrum of exponents, number fields, adelic geometry of numbers, 
parametric geometry of numbers, n-systems, extension of scalars.

\section*{
G\'eom\'etrie param\'etrique des nombres sur un corps de nombres 
et~extension des scalaires}

\subsection*{R\'esum\'e}
La g\'eom\'etrie param\'etrique des nombres de Schmidt et Summerer \'etudie 
l'approxi\-mation rationnelle des points de $\bR^n$.  Nous \'etendons cette th\'eorie 
\`a un corps de nombres $K$ et \`a son compl\'et\'e $K_\pw$ en une place $\pw$ 
pour traiter de l'approximation sur $K$ des points de $K_\pw^n$.  Nous en d\'eduisons 
que les exposants d'approximation sur $\bQ$ des points de $\bR^n$ poss\`edent 
le m\^eme spectre que leurs g\'en\'eralisations sur $K$ dans $K_\pw^n$.  Lorsque 
$\pw$ est de degr\'e relatif \'egal \`a un au-dessus d'une place $\ell$ de $\bQ$, 
nous relions aussi l'approximation sur $K$ d'un point $\uxi$ de $K_\pw^n$ \`a 
celle sur $\bQ$ du point $\Xi$ de $\bQ_\ell^{nd}$ obtenu \`a partir de 
$\uxi$ par extension des scalaires, o\`u $d$ d\'esigne le degr\'e de $K$
sur $\bQ$.  En combinant cette observation \`a un r\'esultat de P.~Bel, nous 
parvenons ainsi \`a construire des courbes alg\'ebriques dans $\bR^{3d}$ 
d\'efinies sur $\bQ$, de degr\'e $2d$, contenant des points qui sont tr\`es 
singuliers vis \`a vis de l'approximation rationnelle.

\subsection*{Mots cl\'es} 
Approximation diophantienne, approximation simultan\'ee, exposants
d'approximation, in\'egalit\'es de transfert, nombres extr\'emaux, points rationnels, 
points singuliers, spectre d'exposants, corps de nombres, g\'eom\'etrie ad\'elique 
des nombres, g\'eom\'etrie param\'etrique des nombres, n-syst\`emes, 
extension des scalaires.

%
%

\section{Introduction}
\label{sec:intro}

In Diophantine approximation, one is interested in measuring how well a given non-zero
point $\uxi\in\bR^n$ with $n\ge 2$ can be approximated by subspaces of $\bR^n$
defined over $\bQ$ of a given dimension $k$.  The most important cases are $k=1$
and $k=n-1$, and each gives rise naturally to a pair of exponents of
approximation.  For $k=1$, they are $\lambdahat(\uxi)$ (resp.\ $\lambda(\uxi)$)
defined as the supremum of all real numbers $\lambda$ for which the inequalities
\begin{equation}
\label{intro:eq:lambda}
 \norm{\ux}\le Q \et \norm{\ux\wedge\uxi}\le Q^{-\lambda}
\end{equation}
have a non-zero solution $\ux\in\bZ^n$ for each large enough $Q\ge 1$ (resp.\ for arbitrarily
large values of $Q\ge 1$).  For $k=n-1$, they are $\omegahat(\uxi)$ (resp.\ $\omega(\uxi)$)
defined as the supremum of all real numbers $\omega$ for which the inequalities
\begin{equation}
\label{intro:eq:omega}
 \norm{\ux}\le Q \et |\ux\cdot\uxi|\le Q^{-\omega}
\end{equation}
have a non-zero solution $\ux\in\bZ^n$ for each large enough $Q\ge 1$ (resp.\ for arbitrarily
large values of $Q\ge 1$), where the dot represents the usual scalar product in $\bR^n$.  This
is independent of the choice of norms in $\bR^n$ and in $\bigwedge^2\bR^n$ but for
convenience, we use the euclidean norms.     As these exponents
depend only on the class of $\uxi$ in $\bP^{n-1}(\bR)$, we may assume that
$\norm{\uxi}=1$.  We refer the reader to the paper of Laurent
\cite{La2009b} for generalizations in intermediate dimensions $k$.

In studying such exponents, it is important to restrict to points $\uxi\in \bR^n$
with $\bQ$-linearly independent coordinates, as this yields simpler statements and
can be achieved by dropping redundant coordinates if necessary.  For such points,
a result of Dirichlet gives
\[
 (n-1)^{-1} \le \lambdahat(\uxi) \le \lambda(\uxi)
 \et
 n-1 \le \omegahat(\uxi) \le \omega(\uxi).
\]
However, this does not fully describe the \emph{spectrum} of
$(\lambdahat,\lambda,\omegahat,\omega)$, namely the set of all quadruples
$(\lambdahat(\uxi),\lambda(\uxi),\omegahat(\uxi),\omega(\uxi))$ associated with
these $\uxi$.
For $n=2$, a complete description is simply given by
\[
 1=\lambdahat(\uxi)=\omegahat(\uxi)\le\lambda(\uxi)=\omega(\uxi)\le \infty
\]
For $n=3$, the description is more complicated and was achieved by Laurent in \cite{La2009a},
showing it as a semi-algebraic set.  One of the constraints that it involves is the following
remarkable identity due to Jarn\'{\i}k \cite[Satz 1]{Ja1938},
\begin{equation}
\label{intro:eq:Jarnik}
\frac{1}{\lambdahat(\uxi)}-1=\frac{1}{\omegahat(\uxi)-1},
\end{equation}
which together with $2\le\omegahat(\uxi)\le\infty$ fully describes the spectrum of the
pair $(\lambdahat,\omegahat)$.  For $n\ge 4$, the spectrum of the four exponents
is not known but Marnat  \cite{Ma2018} has shown that it contains an open subset
of $\bR^4$ and thus it obeys no algebraic relation such as \eqref{intro:eq:Jarnik}.

Many of the recent progresses, including the breakthrough of Marnat and Moshchevitin
\cite{MM2020} who determined the spectra of the pairs $(\lambdahat,\lambda)$ and
$(\omegahat,\omega)$ for each $n\ge 3$, use in a crucial way Schmidt's and
Summerer's parametric geometry of numbers \cite{SS2013}.   In the dual but
equivalent setting of \cite{R2015}, this theory attaches to any point $\uxi\in\bR^n$
with $\norm{\uxi}=1$, the family of symmetric convex bodies of $\bR^n$
\[
 \cC_\uxi(q)
  =\{\ux\in\bR^n\,;\, \norm{\ux}\le 1 \text{ and } |\ux\cdot\uxi|\le e^{-q}\}
  \subseteq \bR^n
\]
parametrized by real numbers $q\ge 0$.  For each $j=1,\dots,n$, let $L_{\uxi,j}(q)$
denote the logarithm of the $j$-th minimum of $\cC_\uxi(q)$ with respect to $\bZ^n$,
namely the smallest real number $t$ such that $e^t\cC_\uxi(q)$ contains at least
$j$ linearly independent points of $\bZ^n$.  Then, form the map
\begin{equation}
\label{intro:eq:Lxi}
 \begin{array}{rcl}
   \uL_\uxi\colon[0,\infty) &\longrightarrow &\bR^n\\
    q &\longmapsto &(L_{\uxi,1}(q),\dots,L_{\uxi,n}(q))\,.
 \end{array}
\end{equation}
Transposed in this setting, the main results of Schmidt and Summerer in \cite{SS2013} 
can be summarized as follows.  Firstly, they note that the 
standard exponents of approximation to $\uxi$, including the four ones mentioned
above, are given by simple formulas in terms of the inferior and superior limits
of the ratios $L_{\uxi,j}(q)/q$ as $q$ goes to infinity.  Secondly, they show
the existence of a constant $\gamma\ge 0$ and of a continuous
piecewise linear map $\uP\colon[0,\infty)\to\bR^n$ with growth conditions
involving $\gamma$, such that the difference $\uL_\uxi-\uP$ is bounded.
Thus the above mentioned exponents of approximation to $\uxi$ can be
computed, via the same formulas, in terms of the behaviour of $\uP$ at infinity.
They call such a map $\uP$ an $(n,\gamma)$-system, and their set increases
as the deformation parameter $\gamma$ increases.  The $(n,0)$-systems, whose
simpler description is recalled in section \ref{sec:results}, are simply called
$n$-systems for shortness.

The main result of \cite{R2015} provides a converse and shows more precisely that the
set of maps $\uL_\uxi$ with $\uxi\in\bR^n$ and $\norm{\uxi}=1$ coincides
with the set of $n$-systems modulo the additive group of bounded functions
from $[0,\infty)$ to $\bR^n$.   Moreover, $\uxi$ has $\bQ$-linearly independent
coordinates if and only if any corresponding $n$-system $\uP=(P_1,\dots,P_n)$
satisfies $\lim_{q\to\infty} P_1(q)=\infty$.  This reduces the determination of
the spectrum of a family of exponents of approximation to a combinatorial
problem about such $n$-systems. 

A similar theory is developped in \cite{RW2017}, with $\bQ$ replaced by a field of 
rational functions in one variable $F(T)$ over an arbitrary field $F$, and $\bR$ replaced 
by the completion $F((1/T))$ of $F(T)$ for the degree valuation.

The first goal of this paper is to extend the theory to a number field $K$ and its
completion $K_\pw$ at a place $\pw$, in order to study approximation over $K$ to an
arbitrary non-zero point $\uxi$ of $K_\pw^n$.  In the next section we show how to attach to
such a point a function $\uL_\uxi\colon[0,\infty)\to\bR^n$ from which the four exponents
of approximation to $\uxi$ can be computed in the same way as in the case where
$K$ is $\bQ$ and $\pw$ is its place at infinity.  We will show that this set of maps
also coincides with the set of $n$-systems modulo bounded functions.  Thus the spectrum
of these exponents remains the same in this new context.   In particular, Jarn\'{\i}k's
identity \eqref{intro:eq:Jarnik} holds for any point $\uxi$ of $K_\pw^3$ with linearly
independent coordinates over $K$.

The second goal of this paper deals with extension of scalars from $\bQ$ to a number field
$K$.  For this we assume that $\pw$ is a place of $K$ of relative degree one over $\bQ$,
so that $K_\pw=\bQ_\ell$ for the place $\ell$ of $\bQ$ induced by $\pw$.  We also choose
a basis $\ualpha=(\alpha_1,\dots,\alpha_d)$ of $K$ over $\bQ$ and for each point
$\uxi=(\xi_1,\dots,\xi_n)\in K_\pw^n$, we define
\begin{equation}
\label{intro:eq:ext}
 \Xi=\ualpha\otimes\uxi=(\alpha_1\uxi,\dots,\alpha_d\uxi)\in K_\pw^{nd}=\bQ_\ell^{nd}
\end{equation}
and say that $\Xi$ is obtained from
$\uxi$ by \emph{extending scalars from $\bQ$ to $K$}.  If $\uxi$ has linearly independent
coordinates over $K$, then $\Xi$ has linearly independent coordinates over $\bQ$ and we
will show a close relationship between the maps $\uL_\uxi$ and $\uL_\Xi$.   From this
we will deduce formulas linking the Diophantine exponents of approximation to
$\uxi$ over $K$ with those of $\Xi$ over $\bQ$.  As a
consequence, we will see that Jarn\'{\i}k's identity \eqref{intro:eq:Jarnik} yields
\begin{equation}
 \label{intro:eq:Jarnik:Xi}
 \frac{1}{\lambdahat(\Xi)}-(2d-1)=\frac{d^2}{\omegahat(\Xi)-(2d-1)}
\end{equation}
for any $\Xi=\ualpha\otimes\uxi\in\bQ_\ell^{3d}$ constructed from a point
$\uxi\in K_\pw^3$ with $K$-linearly independent coordinates.

Let $\ell$ be a place of $\bQ$.  We say that a point $\uxi\in\bQ_\ell^n$ is
\emph{very singular} if it has linearly independent coordinates over $\bQ$
and satisfies $\lambdahat(\uxi)>1/(n-1)$.  This requires $n\ge 3$.   Moreover, by Schmidt's
subspace theorem, such a point is not algebraic and so generates a field
$\bQ(\uxi)$ of transcendence degree at least one over $\bQ$.   The third goal
and the initial motivation of this paper is to provide new examples
of very singular points of transcendence degree one.   Up to now, all known
examples come from dimension $n=3$ and, apart from the constructions of
\cite{Rconic}, they are all of the form $\uxi=(1,\xi,\xi^2)$.  Moreover, the
supremum of $\lambdahat(1,\xi,\xi^2)$ for a transcendental number
$\xi\in\bQ_\ell$ is $1/\gamma\simeq 0.618$ where $\gamma=(1+\sqrt{5})/2$
denotes the golden ratio.  For
$\bQ_\ell=\bR$, this follows from the constructions of \cite{Rnote} or \cite{RcubicI}
together with the upper bound of \cite[Theorem 1a]{DS1969}.
For a prime number $\ell$, this follows from \cite[Chapter 2]{Ze2008} or
\cite{Bu2010} together with \cite[Th\'eor\`eme 2]{Te2002}.
More generally, Bel showed in \cite{Be2013} that the result extends to any
number field $K$ and its completion $K_\pw$ at a place $\pw$.  Assuming that
$\pw$ extends $\ell$ with relative degree $1$ and choosing a basis
$\ualpha=(\alpha_1,\dots,\alpha_d)$ of $K$ over $\bQ$, we will deduce
that $\bQ_\ell^{3d}$ contains very singular points of the form
$(\ualpha,\xi\ualpha,\xi^2\ualpha)$ with $\xi\in\bQ_\ell$.

%
%

\section{Notation and main results}
\label{sec:results}

Throughout this paper, we fix an algebraic extension $K$ of $\bQ$ of finite degree $d$.

\subsection{Absolute values}
We denote by $\MK$ the set
of non-trivial places of $K$, and by $\MKinf$ the subset of its archimedean places.
For each
$\pv\in \MK$, we denote by $K_\pv$ the completion of $K$ at $\pv$ and by
$d_\pv=[K_\pv:\bQ_\pv]$
its local degree.  When $\pv\in \MKinf$, we normalize the absolute value
$|\ |_\pv$ on $K_\pv$ so
that it extends the usual absolute value $|\ |_\infty$ on $\bQ$.  Then $K_\pv$ embeds isometrically
into $\bC$. We identify it with its image $\bR$ or $\bC$, and write $\pv\mid\infty$.  Otherwise,
there is a unique prime number $p$ with $|p|_\pv<1$ and we ask that $|p|_\pv=p^{-1}$ so
that $|\ |_\pv$ extends the usual $p$-adic absolute value on $\bQ$.  
We then write $\pv\mid p$.  For these
normalizations and for each $a\in K^\mult$, the product formula reads
\[
 \prod_{\pv\in \MK} |a|_\pv^{d_\pv/d}=1.
\]

\subsection{Local norms and heights}
\label{results:ssec:norms}
Given a positive integer $n$ and a place $\pv\in \MK$, we define the norm of a point
$\ux=(x_1,\dots,x_n)$ in $K_\pv^n$ by
\[
 \norm{\ux}_\pv = \begin{cases}
   (|x_1|_\pv^2+\cdots+|x_n|_\pv^2)^{1/2} &\text{if $\pv\mid\infty$,}\\[5pt]
   \max\{|x_1|_\pv,\dots,|x_n|_\pv\} &\text{otherwise.}
  \end{cases}
\]
For this choice of local norms, we define the \emph{height} of a  non-zero point
$\ux$ in $K^n$ by
\[
  H(\ux)=\prod_{\pv\in \MK}\norm{\ux}_\pv^{d_\pv/d}.
\]
By the product formula, it depends only on the class of $\ux$ in $\bP^{n-1}(K)$ and
satisfies $H(\ux)\ge 1$.

More generally, for each $k\in\{1,\dots,n\}$ and each $\pv \in \MK$, we define the norm
of a point in $\bigwedge^kK_\pv^n$ to be the norm of its set of Pl\"ucker coordinates
in $K_\pv^N$ where $N=\binom{n}{k}$.  We also define the height of a point in
$\bigwedge^kK^n$ to be the height of its set of Pl\"ucker coordinates
in $K^N$.  This is independent of the ordering of these coordinates.  Then, we
define  the height of a $k$-dimensional subspace $V$ of $K^n$ to be
\[
 H(V)=H(\ux_1\wedge\cdots\wedge\ux_k)
\]
independently of the choice of a basis $(\ux_1,\dots,\ux_k)$ of $V$ over $K$.
For the subspace $0$ of $K^n$, we set $H(0)=1$.

\subsection{The canonical bilinear form}
\label{results:ssec:canbil}
We endow $K^n$ with the bilinear form given by
\begin{equation}
\label{results:eq:bil}
 \ux\cdot\uy=x_1y_1+\cdots+x_ny_n
\end{equation}
for each $\ux=(x_1,\dots,x_n)$ and $\uy=(y_1,\dots,y_n)$ in $K^n$.  Then we define
the orthogonal space to a subspace $V$ of $K^n$ to be
\begin{equation}
\label{results:eq:orth}
 V^\perp=\{\uy\in K^n\,;\, \ux\cdot\uy=0\ \text{for each $\ux\in V$\,}\}.
\end{equation}
According to a result of Schmidt, it has the same height $H(V^\perp)=H(V)$ as $V$.

For each $\pv \in \MK$,  the same formula \eqref{results:eq:bil} provides a bilinear form
on $K_\pv^n$ which we denote in the same way.  For a subspace $V$ of $K_\pv^n$, we also
define $V^\perp$ by \eqref{results:eq:orth} but allowing $\uy$ to run through $K_\pv^n$.

\subsection{Exponents of approximation}
\label{results:ssec:exp}
Fix a place $\pw\in \MK$ and a non-zero point $\uxi\in K_\pw^n$.  For each non-zero
$\ux\in K^n$, we modify slightly the notation of P.~Bel in \cite{Be2013} by setting
\[
 \Dwedge_\uxi(\ux)
  = \Big(\frac{\norm{\ux\wedge\uxi}_\pw}{\norm{\uxi}_\pw}\Big)^{d_\pw/d}
    \prod_{\pv\neq w} \norm{\ux}_\pv^{d_\pv/d}
 \et
 \Dscal_\uxi(\ux)
  = \Big(\frac{|\ux\cdot\uxi|_\pw}{\norm{\uxi}_\pw}\Big)^{d_\pw/d}
    \prod_{\pv\neq w} \norm{\ux}_\pv^{d_\pv/d}\,.
\]
In view of the product formula, these numbers depend only on the class of $\ux$
in $\bP^{n-1}(K)$.  Clearly, they also depend only on the class of $\uxi$ in
$\bP^{n-1}(K_\pw)$.   So, in practice, we may always normalize $\uxi$ so that
$\norm{\uxi}_\pw=1$.

\begin{definition}
\label{results:def:exp}
We denote by $\lambdahat(\uxi,K,\pw)$
(resp.\ $\lambda(\uxi,K,\pw)$) the supremum of all real numbers $\lambda$
for which the inequalities
\[
 H(\ux) \le Q \et \Dwedge_\uxi(\ux)\le Q^{-\lambda}
\]
admit a non-zero solution $\ux\in K^n$ for all sufficiently large
(resp.\ for arbitrarily large) real numbers $Q\ge 1$.  We also denote
by $\omegahat(\uxi,K,\pw)$ (resp.\ $\omega(\uxi,K,\pw)$) the supremum
of all real numbers $\omega$
for which the inequalities
\[
 H(\ux) \le Q \et \Dscal_\uxi(\ux)\le Q^{-\omega}
\]
admit a non-zero solution $\ux\in K^n$ for all sufficiently large
(resp.\ for arbitrarily large) real numbers $Q\ge 1$.
\end{definition}

By construction, these numbers depend only on the class of $\uxi$ in
$\bP^{n-1}(K_\pw)$.  Moreover, when $K=\bQ$ and $\pw=\infty$, these are simply
the standard exponents of approximation to a non-zero point $\uxi\in\bR^n$ from
the introduction.  Indeed, each point of $\bP^{n-1}(\bQ)$ is represented by
a primitive integer point $\ux$, that is a point of $\bZ^n$ with relatively prime
coordinates, and we have $H(\ux)=\norm{\ux}$, $\Dwedge_\uxi(\ux)
 =\norm{\ux\wedge\uxi}$ and $\Dscal_\uxi(\ux) = |\ux\cdot\uxi|$
when $\norm{\uxi}=1$.

We can now state the main result of P.~Bel
in \cite{Be2013} to which we alluded in the introduction.

\begin{theorem}[Bel, 2013]
\label{results:thm:Bel}
Let $\pw\in \MK$, and let $S$ denote the set of elements of
$K_\pw^3$ of the form $\uxi=(1,\xi,\xi^2)$ that have linearly independent
coordinates over $K$.  Then, the supremum of the numbers
$\lambdahat(\uxi,K,\pw)$ with $\uxi\in S$ is $1/\gamma\simeq 0.618$
where $\gamma=(1+\sqrt{5})/2$ stands for the golden ratio.
\end{theorem}

\subsection{Two dual families of minima}
\label{results:ssec:L}
Let $\pw$ and $\uxi\in K_\pw^n$ be as in section \ref{results:ssec:exp}.
For each $j=1,\dots,n$ and each $q\ge 0$, we define $L_{\uxi,j}(q)$ (resp.\
$L^*_{\uxi,j}(q)$\,) to be the smallest real number $t\ge 0$ for which the conditions
\begin{equation}
\label{results:eq:HD}
 H(\ux)\le e^t
 \et
 D_\uxi(\ux)\le e^{t-q} \quad \big( \text{resp.\ } D^*_\uxi(\ux)\le e^{t-q}\,\big)
\end{equation}
admit at least $j$ linearly independent solutions over $K$ in $K^n$.  This minimum
exists since, for any number $B\ge 1$, there are only finitely many elements of
$\bP^{n-1}(K)$ of height at most $B$.  We combine these functions into two maps
\[
  \uL_\uxi=(L_{\uxi,1},\dots,L_{\uxi,n})
  \et
  \uL^*_\uxi=(L^*_{\uxi,1},\dots,L^*_{\uxi,n})
\]
from $[0,\infty)$ to $\bR^n$.  For $K=\bQ$ and $K_\pw=\bR$, the map $\uL_\uxi$
is the same as in the introduction.

\subsection{The $n$-systems}
\label{results:ssec:n-sys}
Let $q_0\in[0,\infty)$.  An $n$-system on $[q_0,\infty)$ is a continuous function
$\uP=(P_1,\dots,P_n)$ from $[q_0,\infty)$ to $\bR^n$ with the following combinatorial
properties.
\begin{itemize}
\item[(S1)] For each $q\in[q_0,\infty)$, we have $0\le P_1(q)\le \cdots\le P_n(q)$
   and $P_1(q)+\cdots+P_n(q)=q$.\\
\item[(S2)] There exist $s\in\{1,2,\dots\}\cup\{\infty\}$ and a strictly increasing
 sequence $(q_i)_{0\le i<s}$ in $[q_0,\infty)$, which is unbounded if $s=\infty$,
 such that, over each subinterval $I_i=[q_{i-1},q_i]$ with $1\le i<s$
 including $I_{s}=[q_{s-1},\infty)$ if $s<\infty$, the union of
 the graphs of $P_1,\dots,P_n$ decomposes as the union of some horizontal line
segments and one line segment $\Gamma_i$ of slope $1$ (with possible crossings), which 
all project down onto $I_i$.\\
\item[(S3)] For each index $i$ with $1\le i<s$, the line segment $\Gamma_i$ ends
strictly above the point where $\Gamma_{i+1}$ starts (on the vertical line with abscissa
$q_i$).
\end{itemize}

The sequence $(q_i)_{0\le i<s}$ is uniquely determined by $\uP$.  Its elements are
called the \emph{switch numbers} of $\uP$.  We say that an $n$-system is
\emph{rigid of mesh} $\mesh$, for a given $\mesh>0$, if the $n$ coordinates of
$\uP(q_i)$ are distinct positive multiples of $\mesh$ for each index $i$ with $0\le i<s$.
Then each $q_i$ is also a positive multiple of $\mesh$ by condition (S1).
See \cite[Figure 1]{R2015} for a picture showing the combined graph of a rigid $5$-system 
with $s=3$.

For each $n$-system $\uP=(P_1,\dots,P_n)\colon[q_0,\infty)\to\bR^n$, we define
its \emph{dual} to be the map $\uP^*\colon[q_0,\infty)\to\bR^n$ given by
\begin{equation}
\label{results:P*}
 \uP^*(q)=\big(q-P_n(q),\,q-P_{n-1}(q),\dots,q-P_1(q)\big)
 \quad
 \text{for each $q\ge q_0$.}
\end{equation}
Note that $\uP^*$ is not an $n$-system unless $n=2$, in which case $\uP^*=\uP$.

\subsection{Main results}  With the notation of \S\ref{results:ssec:L}, we will show
that the main result of parametric geometry of numbers from \cite{R2015} extends
naturally to the present more general setting.  We state it below in dual form as
well.

\begin{thmA}
Let $n\ge 2$ be an integer and let $\pw\in \MK$.  There are constants $c,c'>0$ depending
only on $K$, $\pw$ and $n$ with the following property.  For each non-zero point
$\uxi\in K_\pw^n$, there is an $n$-system $\uP$ on $[0,\infty)$ such that
\begin{equation}
\label{results:thmA:eq}
  \sup_{q\ge 0}\|\uL_\uxi(q)-\uP(q)\|\le c
  \et
  \sup_{q\ge 0}\|\uL^*_\uxi(q)-\uP^*(q)\|\le c
\end{equation}
Conversely, for each $n$-system $\uP$ on $[0,\infty)$, there is a non-zero point
$\uxi\in K_\pw^n$ for which one of the two conditions in \eqref{results:thmA:eq}
holds.  Then the second condition holds with $c$ replaced by $c'$.
\end{thmA}

This means in particular that the two conditions in \eqref{results:thmA:eq} are
equivalent up to the value of the constant $c$.  For $n=1$, the statement of Theorem A 
is also true but not interesting because there is a unique $1$-system 
$\uP$ on $[0,\infty)$ and it satisfies $\uP(q)=\uL_\uxi(q)=q$ and 
$\uP^*(q)=\uL^*_\uxi(q)=0$ for any $q\ge 0$ and any non-zero $\uxi\in K_\pw$.
The next result deals with extension of scalars from $\bQ$ to $K$.

\begin{thmB}
\label{results:thm:extension}
Let $n\ge 2$ be an integer, let $\pw\in \MK$ be a place of $K$ of relative degree $d_\pw=1$ over
a place $\ell$ of $\bQ$, and let $\ualpha\in K^d$ be a basis of $K$ over $\bQ$.  There
is a constant $c''>0$ with the following property.  For each non-zero $\uxi\in K_\pw^n$, the point
$\Xi=\ualpha\otimes\uxi \in\bQ_\ell^{nd}$ (defined in \eqref{intro:eq:ext}) satisfies
\begin{equation}
\label{results:thm:extension:eq}
  |L_{\Xi,d(i-1)+j}(dq)-L_{\uxi,i}(q)| \le c''
  \et
 |L^*_{\Xi,d(i-1)+j}(dq) - L^*_{\uxi,i}(q) - (d-1)q| \le c''
\end{equation}
for any choice of $q\ge 0$, $i=1,\dots,n$ and $j=1,\dots,d$.
\end{thmB}

Again the two sets of conditions in \eqref{results:thm:extension:eq} in terms of the
functions $L$ and $L^*$ are equivalent up to the value of $c''$.  Our last main result
provides very singular points on projective algebraic curves of degree $2d$ defined and
irreducible over $\bQ$.

\begin{thmC}
\label{results:thm:extremal}
Suppose that $K$ embeds in $\bQ_\ell$ for a place $\ell$ of $\bQ$.  Identify $K$ with its
image and choose a basis $\ualpha\in\bQ_\ell^d$ of $K$ over $\bQ$.  Then we have
\[
\begin{aligned}
  &\sup\big\{\, \lambdahat\big((\ualpha,\xi\ualpha,\xi^2\ualpha),\bQ,\ell\big)\,;\,
            \xi\in \bQ_\ell \text{ and } [K(\xi):K]>2\big\}
   =(d\gamma^2-1)^{-1},\\
  &\sup\big\{\, \omegahat\big((\ualpha,\xi\ualpha,\xi^2\ualpha),\bQ,\ell\big)\,;\,
             \xi\in\bQ_\ell \text{ and } [K(\xi):K]>2\big\}
   =d(\gamma^2+1)-1,
\end{aligned}
\]
where $\gamma=(1+\sqrt{5})/2$ stands for the golden ratio.
\end{thmC}

Since $2<\gamma^2<3$, this indeed provides very singular points 
$(\ualpha,\xi\ualpha,\xi^2\ualpha)\in\bQ_\ell^{3d}$.

%
%

\subsection{Outline of the paper}
Most of the paper is devoted to the proof of Theorem A.  This is done in two steps which 
we briefly sketch below.

We first show in section \ref{sec:points} how to attach an $n$-system
to a non-zero point $\uxi\in K_\pw^n$.  The general strategy is similar to that
of Schmidt and Summerer in \cite{SS2013}, instead that we need the adelic versions
of Minkowski's convex body theorem and of Mahler's theory of compound bodies
recalled in section \ref{sec:adelic}.  We also need a notion of distance
$\lambda(\ux,\cC)$ between a non-zero point $\ux$ of $K^n$ and an adelic convex
body $\cC$, and a related notion of adelic minima for $\cC$ defined in section \ref{sec:dilations}.
With those tools, we construct a family of adelic convex bodies $\cC_\uxi(q)$ whose
minima are closely related to the map $\uL_\uxi(q)$, and we obtain information on
this map by considering approximate compounds $\cC^{(k)}_\uxi(q)$ of $\cC_\uxi(q)$.
The existence of an $n$-system that approximates $\uL_\uxi(q)$ up to a bounded function
then follows from a combinatorial result of \cite{R2015} that is recalled in
section \ref{sec:comb}.

The combinatorial result of section \ref{sec:comb} is also used in order to attach
a point $\uxi$ to an $n$-system.  It shows that we simply have to do it for a rigid $n$-system $\uR$
with a large mesh.  The construction of $\uxi$ is done in section \ref{sec:converse}.
We work over the ring $\cO_S$ of $S$-integers of $K$ where $S$ consists of $\pw$
and all archimedean places of $K$.  We construct recursively a sequence of ordered bases
$\ux^{(i)}$ of $\cO_S^n$ over $\cO_S$, one for each of the switch points $q_i$ of $\uR$.
The basis $\ux^{(i)}$ will realize, up to bounded factors, the successive minima of the
adelic convex body $\cC_\uxi(q)$ in the interval between $q_i$ and $q_{i+1}$ for
the point $\uxi$ that we want (as illustrated for example in \cite[Figure 2]{R2015}).   
Each basis $\ux^{(i)}$, except the first, is constructed from
the preceding $\ux^{(i-1)}$ by modifying only one point of it and by moving the new point
up in the sequence, according to the behavior of the map $\uR$ between $q_{i-1}$ 
and $q_i$.  This new point is obtained by multiplying the old one by an appropriate 
$S$-unit and by adding to this product a linear combination of some other points of
$\ux^{(i-1)}$ with coefficients in $\cO_S$, in order to keep control on the
geometry of $\ux^{(i)}$ in $K_\pv^n$ for each place $\pv$ of $S$.  For the places $\pv$
distinct from $\pw$, this is done so that the image of $\ux^{(i)}$ in $K_\pv^n$
remains bounded and almost orthogonal in a sense that is defined in section \ref{sec:metric}.
Meanwhile, at the place $\pw$, the norms of the basis elements of $\ux^{(i)}$ in
$K_\pw^n$ are governed by the coordinates of $\uR(q_i)$, and the subsequence
of $\ux^{(i)}$ common to $\ux^{(i+1)}$ is almost orthogonal.  Then the lines of
$K_\pw^n$ which are orthogonal to these subsequences for the dot product converge
to a line whose generator $\uxi$ has the required property.  The local estimates that
are needed are developed in section \ref{sec:metric}, and the recursive procedure is
presented in section \ref{sec:constr}, together with crucial estimates for local norms
at the place $\pw$ expressed in terms of heights only.

With the help of Theorem A, we show in section \ref{sec:spectra} that the spectrum
of the four exponents $(\omega,\omegahat,\lambda,\lambdahat)$ is independent
of $K$ and $\pw$ and that it can be computed in terms of $n$-systems.  We also
extend the intermediate exponents of Laurent to the number field setting and derive
the same conclusion for their spectrum.

Finally, Theorems B and C are proved in section \ref{sec:proofBC} using a general
construction in adelic geometry of numbers from section \ref{sec:principle} that is
reminiscent of work of Jeff Thunder in \cite{Th2002}.

%
%

\section{Local metric estimates}
\label{sec:metric}

For the sake of generality, we fix here an arbitrary local field $L$, namely a complete field
with respect to an absolute value $|\ |$ which either is archimedean or has a discrete
valuation group $|L^\mult|$ in $\bR^\mult$.  For our applications this will be $K_\pv$
for some place $\pv$ of $K$.  If $L$ is archimedean, we identify it with $\bR$ or $\bC$
through an isometric field embedding in $\bC$ (unique up to composition with complex
conjugation).  Otherwise, we denote by $\cO=\{x\in L\,;\, |x|\le 1\}$ the valuation
ring of $L$.  In this section, we define notions of orthogonality and distance, and provide
several estimates that will be needed in later sections (cf.~\cite[\S4]{R2015}).

\subsection{Norms and orthogonality}
\label{metric:ssec:norms}
Let $k$ and $n$ be integers with $1\le k\le n$, and let $U$ be a vector space over $L$
of dimension $n$.  If $L\subseteq\bC$, we equip $U$ with the euclidean norm associated
to an inner product on $U$ (real if $L=\bR$ and complex if $L=\bC$).  Then there is a
unique inner product on $\tbigwedge^kU$ such that, for any orthonormal basis
$(\uu_1,\dots,\uu_n)$ of $U$, the products $\uu_{i_1}\wedge\cdots\wedge\uu_{i_k}$
with $1\le i_1<\cdots<i_k\le n$ form an orthonormal basis of $\tbigwedge^kU$, and
we equip this space with the associated euclidean norm.  If $L$ is not archimedean,
the ring $\cO$ is a principal ideal domain and we equip $U$ with the maximum norm 
with respect to some basis of $U$ over $L$.  Then, the unit ball $\cB$ for that norm 
is the free rank $n$ sub-$\cO$-module of $U$ generated by this basis and, for each 
$\ux\in U$, we have 
\[
 \norm{\ux}=\min\{ |a| \,;\, a\in L \text{ and } \ux\in a\cB \}.
\]
Moreover, the sub-$\cO$-module $\tbigwedge^k\cB$ of $\tbigwedge^kU$ generated by
the products of $k$ elements of $\cB$ is free of rank $N=\binom{n}{k}$, and we 
equip $\tbigwedge^kU$ with the corresponding norm.

If $V$ is a subspace of $U$, we endow it with the induced norm.  This norm is
admissible because, if $L\not\subseteq \bC$, it is associated to the sub-$\cO$-module 
$\cB\cap V$ of $V$ which is free of rank $\dim_L(V)$.  We say that 
subspaces $V_1,\dots,V_m$ of $V$ are (topologically) \emph{orthogonal} and, following 
the notation of \cite[\S2.2]{RW2017}, we write their sum as
\[
 V_1\ptop\cdots\ptop V_m
\]
if, for any choice of $(\ux_1,\dots,\ux_m)\in V_1\times\cdots\times V_m$, we have
\[
 \norm{\ux_1+\cdots+\ux_m}
 = \begin{cases}
    (\norm{\ux_1}^2+\cdots+\norm{\ux_m}^2)^{1/2} &\text{if $L\subseteq \bC$,}\\
    \max\{\norm{\ux_1}, \dots,\norm{\ux_m}\} &\text{otherwise.}
    \end{cases}
\]
When $L\subseteq \bC$, this is the usual notion and it amounts to asking that
$V_1,\dots,V_m$ are pairwise orthogonal.  However, when $L$ is non-archimedean,
the latter condition is necessary but not sufficient.  We say that a point $\ux\in U$ is
\emph{orthogonal} to a subspace $V$ of $U$ if $\langle\ux\rangle_L$ and $V$ are
orthogonal.
We say that an $m$-tuple of vectors $(\ux_1,\dots,\ux_m)\in U^m$ is \emph{orthogonal}
if the subspaces $\langle\ux_1\rangle_L,\dots,\langle\ux_m\rangle_L$ that they span
are orthogonal.  We say that it is \emph{orthonormal} if moreover they have norm $1$.
Again these are the usual notions when $L\subseteq \bC$.   When $L\not\subseteq\bC$,
an orthonormal basis of $U$ is simply a basis of $\cB$ as an $\cO$-module.  In general,
an $m$-tuple of non-zero vectors $(\ux_1,\dots,\ux_m)$ of $U$ is orthogonal (resp.\
orthonormal) if and only if it can be extended to an orthogonal (resp.\ orthonormal)
basis $(\ux_1,\dots,\ux_n)$ of $U$.  We will also need the following criterion.

\begin{lemma}
\label{metric:lemma:prod}
With the above notation, let $\ux_1,\dots,\ux_m\in U\setminus\{0\}$.  Then,
we have
\[
 \norm{\ux_1\wedge\cdots\wedge\ux_m}
   \le   \norm{\ux_1}\cdots\norm{\ux_m}
\]
with equality if and only if $(\ux_1,\dots,\ux_m)$ is orthogonal.
\end{lemma}

On $L^n$ we have the canonical bilinear form or dot product given by
\eqref{results:eq:bil} for any pair of points $\ux=(x_1,\dots,x_n)$ and
$\uy=(y_1,\dots,y_n)$ in $L^n$.  If $L\subseteq \bC$, this is connected with
 the inner product
\[
 (\ux,\uy):=\ux\cdot\uybar=x_1\ybar_1+\cdots+x_n\ybar_n
\]
where $\uybar=(\ybar_1,\dots,\ybar_n)$ denotes the complex conjugate of $\uy$,
and we equip $L^n$ with the corresponding euclidean norm.  Otherwise, we
equip $L^n$ with the maximum norm, so that its unit ball is $\cO^n$.  When
$L=K_\pv$, this agrees with the definitions of section \ref{results:ssec:norms}
both for the norm on $L^n$ and the corresponding norm on $\tbigwedge^kL^n$.
We conclude with the following observation.

\begin{lemma}
\label{metric:lemma:dual}
Let $(\uu_1,\dots,\uu_n)$ be an orthonormal basis of $L^n$.  The dual
basis $(\uu^*_1,\dots,\uu^*_n)$ of $L^n$ with respect to the dot product is
also orthonormal.
\end{lemma}

\begin{proof}
If $L\subseteq \bC$, then $\uu_j^*$ is the complex conjugate 
$\overline{\uu}_j$ of $\uu_j$ for each $j=1,\dots,n$, thus
$(\uu_i^*,\uu_j^*)=\overline{(\uu_i,\uu_j)}=\delta_{i,j}$ for each
$i,j\in\{1,\dots,n\}$, and we are done.  If $L\not\subseteq \bC$, then
$(\uu_1,\dots,\uu_n)$ is a basis of $\cO^n$ as an $\cO$-module, thus
$(\uu_1^*,\dots,\uu_n^*)$ is also a basis of $\cO^n$ as needed.
\end{proof}


\subsection{Distances}
\label{metric:ssec:dist}
Again, let $1\le k\le n$ be integers and let $U$ be a vector space over $L$
of dimension $n$ equipped with an admissible norm, as above.  By the choice
of norm on $L^n$,
a basis $(\uu_1,\dots,\uu_n)$ of $U$ is orthonormal if and only if the linear
map from $L^n$ to $U$ sending a point $(a_1,\dots,a_n)\in L^n$ to
$a_1\uu_1+\cdots+a_n\uu_n\in U$ is an isometry.   We consider the
following notions of distance.

\begin{definition}
\label{metric:def:dist}
The (projective) distance
between non-zero points $\ux$ and $\uy$ of $U$, or between the lines
$\langle\ux\rangle_L$ and $\langle\uy\rangle_L$ that they generate, is
\[
 \dist(\ux,\uy)
  :=\dist(\langle\ux\rangle_L,\langle\uy\rangle_L)
  :=\frac{\norm{\ux\wedge\uy}}{\norm{\ux}\norm{\uy}} \in [0,1].
\]
If $V_1$, $V_2$ are subspaces of $U$ of the same dimension $k$, then
$\tbigwedge^kV_1$, $\tbigwedge^kV_2$ are one-dimensional subspaces of
$\tbigwedge^kU$, and we define
\[
 \dist(V_1,V_2)=\dist\big(\tbigwedge^kV_1,\tbigwedge^kV_2\big).
\]
\end{definition}

As Schmidt notes in \cite[\S8]{Sc1967}, the distance between non-zero points of $U$
satisfies the triangle inequality when $L\subseteq \bC$.  When $L$ is
non-archimedean, we state below a stronger inequality which we leave to the reader.

\begin{lemma}
\label{metric:lemma:triangle}
For any non-zero points $\ux_1,\ux_2,\ux_3\in U$, we have
\[
 \dist(\ux_1,\ux_3)
  \le\begin{cases}
       \dist(\ux_1,\ux_2)+\dist(\ux_2,\ux_3) &\text{if $L\subseteq\bC$,}\\
       \max\{\dist(\ux_1,\ux_2),\,\dist(\ux_2,\ux_3)\} &\text{else.}
       \end{cases}
\]
The same holds if $\ux_j$ is replaced by a subspace $V_j$ of $U$ of
dimension $k\ge 1$ for $j=1,2,3$.
\end{lemma}

The second assertion of the lemma follows from the first when $k=1$. The general
case where $k>1$ follows by considering the lines $\tbigwedge^k V_j$ inside
$\tbigwedge^k U$.

For any subspace $V$ of $U$, there is a subspace $W$ of $U$ such that
$U=W\ptop V$.  It suffices to choose an orthonormal basis $(\uu_1,\dots,\uu_k)$
for $V$ (empty if $V=0$), to complete it to an orthonormal basis
$(\uu_1,\dots,\uu_n)$ of $U$, and to take $W=\langle\uu_{k+1},\dots,
\uu_n\rangle_L$.  So, we may write any $\ux\in U$ in the form $\ux=\uw+\uv$
with $\uw\in W$ orthogonal to $V$ and $\uv\in V$.  If $L\subseteq \bC$, this
decomposition is unique.  In general, it is not unique but the next result shows
that $\norm{\uw}$ is independent of the decomposition (upon noting that $\uw=0$
when $\ux=0$).

\begin{lemma}
\label{metric:lemmaC}
Let $V$ be a non-zero subspace of $U$ and let $\ux \in U\setminus\{0\}$.
Write $\ux=\uw+\uv$ with $\uw$ orthogonal to $V$ and $\uv\in V$.  Then we
have
\begin{equation}
 \label{metric:lemmaC:eq}
 \dist(\ux,V)
   := \min\big\{\dist(\ux,\uy)\,;\,\uy\in V\setminus\{0\}\big\}
   = \frac{\norm{\uw}}{\norm{\ux}}.
\end{equation}
\end{lemma}

\begin{proof}
Let $\uy\in V\setminus\{0\}$.  Since $\uw$ is orthogonal to $V$, it is orthogonal
to $\uy$.  So, the products $\uw\wedge\uy$ and $\uv\wedge\uy$ are orthogonal
in $\bigwedge^2U$.  Consequently, we find
\[
 \dist(\ux,\uy)
  = \frac{\norm{\uw\wedge\uy+\uv\wedge\uy}}{\norm{\ux}\,\norm{\uy}}
  \ge \frac{\norm{\uw\wedge\uy}}{\norm{\ux}\,\norm{\uy}}
  = \frac{\norm{\uw}}{\norm{\ux}},
\]
with equality everywhere if $\uv=0$ or if $\uy=\uv\neq 0$.
\end{proof}

By Lemma \ref{metric:lemma:prod},  non-zero points $\ux$, $\uy$ of $U$
are orthogonal if and only if $\dist(\ux,\uy)=1$.  Thus with the notation of Lemma
\ref{metric:lemmaC}, the point $\ux$ is orthogonal to $V$ if and only if
$\dist(\ux,V)=1$.  Moreover, we have $\ux\in V$ if and only if $\dist(\ux,V)=0$.
We also note the following alternative formula for $\dist(\ux,V)$.

\begin{lemma}
\label{metric:lemmaCbis}
Let $\ux$ and $V$ be as in Lemma \ref{metric:lemmaC} and let $(\uy_1,\dots,\uy_k)$
be a basis of $V$.  Then we have
\begin{equation}
 \label{metric:lemmaCbis:eq}
 \dist(\ux,V)
   = \frac{\norm{\ux\wedge\uy_1\wedge\cdots\wedge\uy_k}}%
      {\norm{\ux}\norm{\uy_1\wedge\cdots\wedge\uy_k}}.
\end{equation}
\end{lemma}

\begin{proof}
Since the right-hand side of \eqref{metric:lemmaCbis:eq} is independent of the
choice of $(\uy_1,\dots,\uy_k)$, we may assume that this basis is orthogonal.
Then, for the decomposition $\ux=\uw+\uv$ of Lemma \ref{metric:lemmaC},
the sequence $(\uw,\uy_1,\dots,\uy_k)$ is also orthogonal.  Thus, using Lemma
\ref{metric:lemma:prod}, we find
\[
 \norm{\ux\wedge\uy_1\wedge\cdots\wedge\uy_k}
  = \norm{\uw\wedge\uy_1\wedge\cdots\wedge\uy_k}
  = \norm{\uw}\,\norm{\uy_1\wedge\cdots\wedge\uy_k}
\]
and so the right-hand side of \eqref{metric:lemmaCbis:eq}  reduces to
$\norm{\uw}/\norm{\ux}=\dist(\ux,V)$.
\end{proof}

For the next crucial lemma, we apply the previous results with $U=L^n$.

\begin{lemma}
\label{metric:lemmaD}
Suppose that $n\ge m\ge 2$ for an integer $m$.  Let $V_1$, $V_2$ be subspaces
of $L^n$ of dimension $m-1$ for which $W=V_1\cap V_2$ has dimension at
least $m-2$.  Then we have
\begin{equation}
\label{metric:lemmaD:eq1}
 \dist(V_1,V_2)=\max\big\{\dist(\ux,V_2)\,;\,\ux\in V_1\setminus\{0\}\,\big\}.
\end{equation}
Moreover, if $V_1\neq V_2$, if $(\uw_1,\dots,\uw_{m-2})$ is a basis of $W$,
and if $\uv_i\in V_i\setminus W$ for $i=1,2$, then upon writing
$\omega=\uw_1\wedge\cdots\wedge\uw_{m-2}$ we have
\begin{equation}
\label{metric:lemmaD:eq2}
 \dist(V_1,V_2)
   =\frac{\norm{\omega}\,\norm{\omega\wedge\uv_1\wedge\uv_2}}%
      {\norm{\omega\wedge\uv_1}\,\norm{\omega\wedge\uv_2}}.
\end{equation}
\end{lemma}

\begin{proof}
We may assume that $V_1\neq V_2$ since otherwise both sides of
\eqref{metric:lemmaD:eq1} are zero.  We also note that the right-hand side of
\eqref{metric:lemmaD:eq2} is independent of the choice of $\uw_1,\dots,\uw_{m-2},
\uv_1,\uv_2$.  So, we choose for $(\uw_1,\dots,\uw_{m-2})$ an orthonormal
basis of $W$ and we complete it to an orthonormal basis
$(\uw_1,\dots,\uw_{m-2},\uv_2,\uu)$ of $V_1+V_2$ with $\uv_2\in V_2$.  We also
choose a unit vector $\uv_1\in V_1$ of the form $\uv_1=a\uv_2+b\uu$ with
$(a,b)\in L^2$.  Then  $(\uw_1,\dots,\uw_{m-2},\uv_i)$ is an orthonormal
basis of $V_i$ for $i=1,2$ and we have $\norm{(a,b)}=1$.  Moreover, the pair
$(\omega\wedge\uv_2,\omega\wedge\uu)$ is orthonormal in $\tbigwedge^{m-1}L^n$.
Since $\omega\wedge\uv_1=a\omega\wedge\uv_2+b\omega\wedge\uu$, we
deduce that
\[
 \dist(V_1,V_2)
  = \dist(\omega\wedge\uv_1,\omega\wedge\uv_2) = |b|.
\]
This proves \eqref{metric:lemmaD:eq2} since
$\norm{\omega\wedge\uv_1\wedge\uv_2}=\norm{b\omega\wedge\uu\wedge\uv_2}
=|b|$.  Finally, let $\ux\in V_1\setminus \{0\}$, and write $\ux=\uw+t\uv_1$
with $\uw\in W$ and $t\in L$.  Then, $\ux=tb\uu+\uv$ where $\uv=\uw+ta\uv_2
\in V_2$ and $\uu$ is orthogonal to $V_2$.  So, Lemma \ref{metric:lemmaC} gives
\[
 \dist(\ux,V_2)=\frac{|tb|}{\norm{\ux}}\le\frac{|tb|}{|t|}=|b|
\]
with equality if $\ux=\uv_1$.  This proves \eqref{metric:lemmaD:eq1}.
\end{proof}

\begin{cor}
\label{metric:cor:lemmaD}
Let $V_1$, $V_2$ be as in Lemma \ref{metric:lemmaD} and let $\ux\in L^n\setminus\{0\}$.  
Then we have
\[
 \dist(\ux,V_2)\le
  \begin{cases}
    \dist(\ux,V_1)+\dist(V_1,V_2) &\text{if $L\subseteq \bC$,}\\
    \max\{\dist(\ux,V_1),\dist(V_1,V_2) \} &\text{else.}
   \end{cases}
\]
\end{cor}

\begin{proof}
Choose $\uy_1\in V_1\setminus\{0\}$ such that $\dist(\ux,\uy_1)=\dist(\ux,V_1)$
and $\uy_2\in V_2\setminus\{0\}$ such that $\dist(\uy_1,\uy_2)=\dist(\uy_1,V_2)$.
By Lemma \ref{metric:lemmaD}, we have $\dist(\uy_1,\uy_2)\le\dist(V_1,V_2)$.
As $\dist(\ux,V_2)\le \dist(\ux,\uy_2)$, the conclusion follows from the triangle
inequality of Lemma \ref{metric:lemma:triangle} applied to $\ux$, $\uy_1$ and
$\uy_2$.
\end{proof}


\subsection{Duality}
\label{metric:ssec:duality}
For each $k=0,\dots,n$, the dot product on $L^n=\tbigwedge^1L^n$ induces a
non-degenerate bilinear map from $\tbigwedge^kL^n\times\tbigwedge^kL^n$ to $L$
also denoted by a dot and given on pure products by
\[
 (\ux_1\wedge\cdots\wedge\ux_k)\cdot(\uy_1\wedge\cdots\wedge\uy_k)
  = \det(\ux_i\cdot\uy_j)
\]
with the convention that, for $k=0$, the empty wedge product is $1\in
L=\tbigwedge^0L^n$ and the empty determinant is $1$ as well.

Let $(\ue_1,\dots,\ue_n)$ denote the canonical basis of $L^n$ and let
$\uE=\ue_1\wedge\cdots\wedge\ue_n$.  For $k$ as above, there
is a unique isomorphism $\varphi_k\colon\tbigwedge^kL^n\to\tbigwedge^{n-k}L^n$
such that
\[
 (\uX\wedge\uY)\cdot\uE=\varphi_k(\uX)\cdot\uY
\]
for any $\uX\in\tbigwedge^kL^n$ and $\uY\in\tbigwedge^{n-k}L^n$.
If $(\uu_1,\dots,\uu_n)$ is any basis of $L^n$ with
$\uu_1\wedge\cdots\wedge\uu_n=\uE$, and if $(\uu^*_1,\dots,\uu^*_n)$
denotes the dual basis of $L^n$ for the dot product, a short computation shows that,
for any $k$-tuple of integers $\ui=(i_1,\dots,i_k)$ with $1\le i_1<\cdots<i_k\le n$,
we have
\begin{equation}
\label{metric:duality:eq}
 \varphi_k(\uu_{i_1}\wedge\cdots\wedge\uu_{i_k})
   = \epsilon(\ui,\uj) \uu^*_{j_1}\wedge\cdots\wedge\uu^*_{j_{n-k}}
\end{equation}
where $\uj=(j_1,\dots,j_{n-k})$ denotes the complementary
increasing sequence of integers for which
$(\ui,\uj)$ is a permutation of $(1,\dots,n)$, and $\epsilon(\ui,\uj)
\in\{-1,1\}$ is the signature of this permutation.  In particular, if we choose
$(\uu_1,\dots,\uu_n)$ to be the canonical basis of $L^n$, which is its own
dual, this formula shows that $\varphi_k$ is an isometry.  Furthermore, if
$V$ is a subspace of $L^n$ of dimension $k$, we may choose $(\uu_1,\dots,\uu_n)$
so that $(\uu_1,\dots,\uu_k)$ is a basis of $V$.  Then
$(\uu^*_{k+1},\dots,\uu^*_n)$ is a basis of the subspace
\[
 V^\perp=\{\uy\in L^n\,;\, \ux\cdot\uy=0 \text{ for each $\ux\in V$}\}
\]
and the same formula implies that
\[
 \varphi_k(\tbigwedge^kV)=\tbigwedge^{n-k}V^\perp.
\]
As $\varphi_k$ is an isometry, it preserves the distance and so we conclude as
follows.

\begin{lemma}
\label{metric:lemma:distdual}
For any pair of subspaces $V_1$, $V_2$ of $L^n$ of the same dimension $k$
with $0<k<n$, we have
\[
 \dist(V_1,V_2)=\dist(V_1^\perp,V_2^\perp).
\]
\end{lemma}

In particular, if $V_1$, $V_2$ have dimension $n-1>0$ and if
$V_i^\perp=\langle\uu_i\rangle_L$ for $i=1,2$, then
$\dist(V_1,V_2)=\dist(\uu_1,\uu_2)$.   When $L=\bR$, this observation also
follows from \cite[Lemma 4.4]{R2015}.


\subsection{Almost orthogonal sequences}
\label{metric:ssec:ao}
We set
\begin{equation}
\label{metric:delta}
 \delta
  = \begin{cases}
       1 &\text{if $L\subseteq\bC$,}\\
       0 &\text{otherwise,}
     \end{cases}
\end{equation}
and say that a non-empty sequence $(\ux_1,\dots,\ux_m)$ in $\Kv^n$ is
\emph{almost orthogonal} if it is linearly independent over $\Kv$ and satisfies
\[
 \dist(\ux_j, \langle\ux_1,\dots,\ux_{j-1}\rangle_L)\ge 1-\delta/2^{j-1}
\quad (2\le j\le m).
\]

Thus almost orthogonal means orthogonal when $L$ is non-archimedean.
As in \cite[\S4]{R2015}, we note that any non-empty subsequence of an almost
orthogonal sequence is almost orthogonal.  Since
$\prod_{j\ge 2} (1-\delta/2^{j-1})\ge e^{-2\delta}$, we find
the following estimate (cf.\ \cite[Lemma 4.6]{R2015}).

\begin{lemma}
\label{metric:lemmaQ}
For any almost orthogonal sequence $(\ux_1,\dots,\ux_m)$ in $\Kv^n$ we have
\[
 e^{-2\delta}\norm{\ux_1}\cdots\norm{\ux_m}
  \le \norm{\ux_1\wedge\cdots\wedge\ux_m}
  \le \norm{\ux_1}\cdots\norm{\ux_m}.
\]
\end{lemma}

The next crucial result is analogous to \cite[Lemma 4.7]{R2015}.  It uses 
the convention that a hat on an element of a sequence or product means 
that this element is omitted from the sequence or product.

\begin{lemma}
\label{metric:lemmaV}
Let $k$, $\ell$, $m$ be integers with $1\le k<\ell\le m\le n$ and let $\uy_1,\dots,\uy_m$
be linearly independent points of $\Kv^n$.  Suppose that the sequences
$(\uy_1,\dots,\widehat{\uy_\ell\,},\dots,\uy_m)$ and
$(\uy_1,\dots,\widehat{\uy_k},\dots,\uy_m)$ are both almost orthogonal.  Then,
the subspaces
\[
 V_1=\langle\uy_1,\dots,\widehat{\uy_\ell\,},\dots,\uy_m\rangle_{\Kv}
 \et
 V_2=\langle\uy_1,\dots,\widehat{\uy_k\,},\dots,\uy_m\rangle_{\Kv}
\]
that they span in $\Kv^n$ satisfy
\[
 \dist(V_1,V_2) \le e^{4\delta}\frac{\norm{\uy_1\wedge\cdots\wedge\uy_m}}{\norm{\uy_1}\cdots\norm{\uy_m}}.
\]
\end{lemma}

\begin{proof}
Upon setting $\omega=\uy_1\wedge\cdots\wedge\widehat{\uy_k}\wedge\cdots
\wedge\widehat{\uy_\ell\,}\wedge\cdots\wedge\uy_m$, Lemma \ref{metric:lemmaD}
gives
\[
 \dist(V_1,V_2)
  = \frac{\norm{\omega}\,  \norm{\uy_1\wedge\cdots\wedge\uy_m}}%
             {\norm{\omega\wedge\uy_\ell}\,  \norm{\omega\wedge\uy_k}}.
\]
The conclusion follows because, by Lemma \ref{metric:lemmaQ},
\[
\begin{aligned}
 \norm{\omega\wedge\uy_\ell}\,  \norm{\omega\wedge\uy_k}
     &\ge e^{-4\delta}
      \big(\norm{\uy_1}\cdots\widehat{\norm{\uy_k}}\cdots\norm{\uy_m}\big)
      \big(\norm{\uy_1}\cdots\widehat{\norm{\uy_\ell}}\cdots\norm{\uy_m}\big)\\
    &\ge e^{-4\delta}\norm{\omega}\, \big(\norm{\uy_1}\cdots\norm{\uy_m}\big).
\end{aligned}
\qedhere
\]
\end{proof}

We conclude with a simple estimate.

\begin{lemma}
\label{metric:lemma:pscal}
For any unit vectors $\uu,\uu'\in L^n$ and any $\ux\in L^n$, we have
\begin{equation}
\label{metric:eq:dot}
 |\ux\cdot\uu|\le 2^\delta \max\{ |\ux\cdot\uu'|,\ \norm{\ux} \dist(\uu,\uu')\}.
\end{equation}
\end{lemma}

\begin{proof}
We have
$\norm{(\ux\cdot\uu)\uu'-(\ux\cdot\uu')\uu}\le\norm{\ux}\,\norm{\uu\wedge\uu'}$
for any $\ux, \uu, \uu' \in L^n$.
\end{proof}

When $\ux\cdot\uu'=0$, this becomes simply
$|\ux\cdot\uu|\le 2^\delta \norm{\ux} \dist(\uu,\uu')$.


%
%


\section{Adelic geometry of numbers}
\label{sec:adelic}

For each place $\pv$ of $K$, we form the compact set
\[
 \cOv=\{x\in K_\pv\,;\, |x|_\pv\le 1\}.
\]
When $\pv\nmid\infty$, this is the ring of integers of $K_\pv$, and we
follow MacFeat \cite{Mc1971} in normalizing the Haar measure
$\mu_\pv$ on $K_\pv$ so that $\mu_\pv(\cOv)=1$.  When $\pv\mid\infty$,
we set $\mu_\pv$ to be the Lebesgue measure on $K_\pv$, with
$\mu_\pv(\cOv)=2$ if $K_\pv=\bR$, and $\mu_\pv(\cOv)=\pi$ if $K_\pv=\bC$.
We also denote by $\mu_\pv$ the product measure on $K_\pv^n$.

The ring of ad\`eles of $K$ is the subring $K_\bA$ of $\prod_{\pv\in \MK}K_\pv$
which consists of the sequences $(a_\pv)$ with $a_\pv\in\cOv$ for all
but finitely $\pv$.  It is endowed with the unique topology which extends the
product topology on the set $\cO_\bA:=\prod_{\pv\mid\infty}K_\pv\times
\prod_{\pv\nmid\infty}\cOv$, and makes $K_\bA$ into a locally compact ring
with $\cO_\bA$ as an open subring.  Then $K$ embeds in $K_\bA$ as a discrete subring
via the diagonal map.  We denote by $\mu$ the Haar measure on $K_\bA$ whose
restriction to $\cO_\bA$ is the product of the $\mu_\pv$, and we use the same
notation for the product measure on $K_\bA^n$.

When $\pv\mid\infty$, a (Minkowski) convex body of $K_\pv^n$ is
any compact convex neighborhood $\cCv$ of $0$ such that
$a\,\cCv\subseteq \cCv$ for each $a\in \cOv$.  When $\pv\nmid \infty$,
this is any finitely generated (thus free and compact)
$\cOv$-submodule $\cCv$ of $K_\pv^n$ of rank $n$. Finally, a convex
body of $K_\bA^n$ is any product $\cC=\prod_\pv\cCv$
where $\cCv$ is a convex body of $K_\pv^n$ for each $\pv$, and
$\cCv=\cOv^n$ for all but finitely $\pv$.
Then the induced topology on $\cC$ coincides with the usual product topology,
and its volume $\mu(\cC)=\prod_\pv\mu_\pv(\cCv)$ is finite and positive.

For each $j=1,\dots,n$, we define the $j$-th minimum $\lambda_j(\cC)$
of a convex body $\cC=\prod_\pv\cCv$ of $K_\bA^n$ to be the
smallest $\lambda>0$ for which the dilated convex body
\[
 \lambda\,\cC=\prod_{\pv\mid \infty}\big(\lambda\,\cCv\big) \prod_{\pv\nmid\infty}\cCv
\]
contains at least $j$ linearly independent elements of $K^n$
over $K$.  With this notation, the adelic version of Minkowski's theorem
reads as follows \cite{Mc1971, BV1983}.

\begin{theorem}[McFeat, 1971; Bombieri and Vaaler, 1983]
 \label{adelic:thm:MBV}
For any convex body $\cC$ of $K_\bA^n$, we have
\[
  \left(\lambda_1(\cC)\cdots\lambda_n(\cC)\right)^d\mu(\cC)
  \asymp 1,
\]
with implicit constants that depend only on $K$ and $n$.
\end{theorem}

We refer the reader to \cite[Theorems 5 and 6]{Mc1971},
\cite[Theorems 3 and 6]{BV1983} and \cite[Corollary of Theorem 1]{Th2002}
for explicit lower bounds and upper bounds.  In particular
this result implies that, if the volume $\mu(\cC)$ of $\cC$ is
large enough, then $\lambda_1(\cC)\le 1$ and so $\cC$ contains
a non-zero point of $K^n$.

More generally, fix an integer $k$ with $1\le k\le n$ and set $N=\binom{n}{k}$.
The $K$-linear isomorphism from $\tbigwedge^kK^n$ to $K^N$ that sends a
point to its Pl\"ucker coordinates extends to a $K_\pv$-linear topological isomorphism
from $\tbigwedge^kK_\pv^n$ to $K_\pv^N$ and to a $K_\bA$-linear topological
isomorphism from $\tbigwedge^kK_\bA^n$ to $K_\bA^N$.  Identifying these pairs
of spaces, we obtain a measure $\mu_\pv$ on $\tbigwedge^kK_\pv^n$ for each
$\pv\in \MK$ and a measure $\mu$ on $\tbigwedge^kK_\bA^n$.  This also provides the
notion of a (Minkowski) convex body $\cK_\pv$ of $\tbigwedge^kK_\pv^n$ for each
$\pv\in \MK$ and of a convex body $\cK=\prod_\pv\cK_\pv$ of $\tbigwedge^kK_\bA^n$,
as well as the notion of the $j$-th minimum $\lambda_j(\cK)$ of $\cK$ with respect
to $\tbigwedge^kK^n$, for each $j=1,\dots,N$.

Let $\cC=\prod_\pv\cCv$ be a convex body of $K_\bA^n$.  Its $k$-th compound
is the convex body $\bigwedge^k\cC=\prod_\pv(\bigwedge^k\cCv)$ of
$\bigwedge^k K_\bA^n$  whose component $\bigwedge^k\cCv$ at a place $\pv$
is the smallest Minkowski convex body of $\bigwedge^kK^n_\pv$ containing all products
$\ux_1\wedge\cdots\wedge\ux_k$ of $k$ elements $\ux_1,\dots,\ux_k$ of $\cCv$.
In particular, we have $\bigwedge^1\cC=\cC$.  In this context, E.~B.~Burger
has extended Mahler's theory of compound bodies in \cite{Bu1993}.  Leaving out
the explicit values of the constants from \cite[Theorem 1.2]{Bu1993}, he showed that
the minima of these convex bodies are related as follows.

\begin{theorem}[Burger, 1993]
 \label{adelic:thm:Burger}
With the above notation, order the $N$ products
$\lambda_{i_1}(\cC)\cdots\lambda_{i_k}(\cC)$ with $1\le i_1<\cdots<i_k\le n$
into a monotonically increasing sequence $\Lambda_1\le\cdots\le\Lambda_N$.  Then,
for each $j=1,\dots,N$, we have
\[
  \lambda_j\left(\tbigwedge^k\cC\right) \asymp \Lambda_j
\]
with implicit constants depending only on $K$ and $n$.
\end{theorem}

We note that $\Lambda_1=\lambda_1(\cC)\cdots\lambda_k(\cC)$,
and that $\Lambda_2=\lambda_1(\cC)\cdots\lambda_{k-1}(\cC)\lambda_{k+1}(\cC)$
if $k<n$.  Moreover, if $\ux_1,\dots,\ux_n$ are linearly independent elements of $K^n$ which
realize the successive minima of $\cC$ in the sense that $\ux_i\in\lambda_i(\cC)\cC$
for $i=1,\dots,n$, then $\uX=\ux_1\wedge\cdots\wedge\ux_k$ belongs to
$\Lambda_1\tbigwedge^k\cC$.  Thus, by the above theorem, the first
minimum of $\tbigwedge^k\cC$ is realized up to a bounded factor by the
pure product $\uX$.

In practice, the compounds of a given convex body are
difficult to compute exactly.  So, we instead use approximations of them,
like in the standard theory (see \cite[Chapter IV, \S7]{Sc1980}).


%
%

\section{Dilations}
\label{sec:dilations}

The group of id\`eles of $K$ is the group $K_\bA^\mult$ of
invertible elements of $K_\bA$.  It contains the multiplicative
group $K^\mult$ of $K$ as a subgroup.  We define the module
$\mua$ of an id\`ele $\ua=(a_\pv) \in K_\bA^\mult$ by
\[
  \mua=\prod_{\pv\in \MK}|a_\pv|_\pv^{d_\pv/d} \in\bR_{>0},
\]
and recall that $\module{\alpha}=1$ for any $\alpha\in K^\mult$.
Then for each convex body $\cC=\prod_\pv\cCv$ of $K_\bA^n$, the
product
\[
 \ua\,\cC=\prod_{\pv\in \MK}(a_\pv\,\cCv)
\]
is a convex body of volume $\mu(\ua\,\cC)=\mua^{dn}\mu(\cC)$.
This construction extends the definition of $\lambda\cC$
with $\lambda\in\bR_{>0}$ by identifying any such $\lambda$
with the id\`ele having component $\lambda$ at each archimedean
place $\pv\mid\infty$ and component $1$ at all other places.

\begin{definition}
\label{dilations:def}
For $\cC$ as above and for each non-zero $\ux\in K^n$, we set
\begin{equation}
\label{dilations:eq:lambdaxC}
 \lambda(\ux,\cC)=\min\{\mua\,;\,\ua\in K_\bA^\mult \text{ and } \ux\in \ua\cC\}.
\end{equation}
For each $j=1,\dots,n$, we also define $\lambda_j^\bA(\cC)$ to be the
smallest $\lambda>0$ for which there are at least $j$ linearly independent
elements $\ux$ of $K^n$ with $\lambda(\ux,\cC)\le \lambda$.
\end{definition}

The minimum exists in \eqref{dilations:eq:lambdaxC} because for each $\pv\in \MK$
there is a non-zero $a_\pv\in K_\pv$ with $|a_\pv|_\pv$ minimal such that $\ux\in a_\pv\cCv$,
and we may choose $a_\pv=1$ for all but finitely many $\pv$.
Moreover, by the product formula, the value $\lambda(\ux,\cC)$, which we view as
a sort of distance from $\ux$ to $\cC$, depends only on the class of $\ux$ in
$\bP^{n-1}(K)$.  In particular, it is independent of $\ux$ if $n=1$.  We also note
that, for any given $t>0$, the non-zero points $\ux$ of $K^n$ with
$\lambda(\ux,\cC)\le t$ have height at most $ct$ for a constant $c>0$ depending
only on $\cC$.  So these points $\ux$ belong to finitely many classes in $\bP^{n-1}(K)$
and for them $\lambda(\ux,\cC)$ takes finitely many values in $[0,t]$.  Hence, there
is a basis $(\ux_1,\dots,\ux_n)$ of $K^n$ over $K$ such that $\lambda_j^\bA(\cC)
=\lambda(\ux_j,\cC)$ for each $j=1,\dots,n$.  In particular, it is sensible to define
each $\lambda_j^\bA(\cC)$ as a minimum.

To compare these minima to those of MacFeat and Bombieri--Vaaler, we need
the following special case of the strong approximation theorem from
\cite[Theorem 3]{Ma1964}.

\begin{lemma}
\label{dilations:lemma:strongapp}
There exists a constant $c_1=c_1(K)>0$ with the following
property.  For each $\ua=(a_\pv)\in K_\bA^\mult$ with
$\mua\ge c_1$, there exists $\alpha\in K^\mult$ such
that $|\alpha|_\pv\le |a_\pv|_\pv$ for each $\pv\in \MK$.
\end{lemma}

Note that this also follows from the adelic version of Minkowski's theorem,
because, for given $\ua=(a_\pv)\in K_\bA^\mult$, the set of points
$(x_\pv)\in K_\bA$ with $|x_\pv|_\pv\le |a_\pv|_\pv$ for all $\pv$ is the convex body
$\ua\,\cB$ of $K_\bA$ of volume $\mua^d\mu(\cB)$, where $\cB=\prod_\pv\cO_\pv$.
So, if $\mua$ is large enough, Theorem \ref{adelic:thm:MBV} gives $\lambda_1(\ua\,\cB)\le 1$,
and thus $\ua\,\cB$ contains some non-zero element of $K$.

\begin{proposition}
 \label{dilations:prop}
Let $\cC$ be a convex body of $K_\bA^n$ and let $j\in\{1,\dots,n\}$.
Then, we have
\begin{equation}
\label{dilations:prop:eq1}
 c_1^{-1}\lambda_j(\cC)
 \le \lambda^\bA_j(\cC)
 \le \lambda_j(\cC)
\end{equation}
where $c_1$ comes from Lemma \ref{dilations:lemma:strongapp}.  Moreover, for
each id\`ele $\ua\in K_\bA^\mult$, we also have
\begin{equation}
\label{dilations:prop:eq2}
 \lambda^\bA_j(\ua\,\cC)=\mua^{-1}\lambda^\bA_j(\cC).
\end{equation}
\end{proposition}

\begin{proof}
Set $\lambda=\lambda^\bA_j(\cC)$ and choose a set $F$ of $j$ linearly
independent points $\ux$ of $K^n$ with $\lambda(\ux,\cC)\le \lambda$.
Given $\ux\in F$, there exists $\ua\in K_\bA^\mult$ with $\mua\le
\lambda$ such that $\ux\in\ua\cC$.  As the id\`ele $\ua' = (a'_\pv)  :=
c_1\lambda\ua^{-1}$ satisfies $\module{\ua'}\ge c_1$, Lemma
\ref{dilations:lemma:strongapp} provides $\alpha\in K^\mult$ such that
 $|\alpha|_\pv\le |a'_\pv|_\pv$ for each $\pv\in \MK$, and then the point $\alpha\ux$
of $K^n$ belongs to $\alpha\ua\cC \subseteq \ua'\ua\cC=c_1\lambda\cC$.
Doing this for each $\ux\in F$, we obtain $j$ linearly independent points
of $K^n$ in $c_1\lambda\cC$.  This means that $\lambda_j(\cC)\le c_1\lambda$,
which amounts to the first inequality in \eqref{dilations:prop:eq1}.

To prove the second inequality in \eqref{dilations:prop:eq1}, set $\lambda
= \lambda_j(\cC)$.  Then $\lambda\cC$ contains at least $j$ linearly
independent elements of $K^n$.  As $\module{\lambda}=\lambda$, this
implies that $\lambda^\bA_j(\cC)\le\lambda$ and we are done.

Finally, \eqref{dilations:prop:eq2} follows from the definitions and the
multiplicativity of the module on $K_\bA^\mult$.
\end{proof}

In view of our identifications (see section \ref{sec:adelic}), the above results and
definitions apply with $\cC$ replaced by any convex body $\cK=\prod_\pv\cK_\pv$ of
$\tbigwedge^k K_\bA^n$ for any integer $1\le k\le n$, provided that $n$ is replaced by
$N=\binom{n}{k}$ and that $K^n$ is replaced by $\tbigwedge^k K^n$.

%
%

\section{A combinatorial result}
\label{sec:comb}

In preparation for the proof of Theorem A in the next sections, we will need
the following result from \cite{R2015}.  We refer the reader to section
\ref{results:ssec:n-sys} for the definition of an $n$-system.

\begin{proposition}
\label{comb:prop}
Let $c\ge 0$.  Suppose that, for each $k=1,\dots,n$,  there are continuous
functions $L_k\colon[0,\infty)\to\bR$ and $M_k\colon[0,\infty)\to\bR$ which are
piecewise linear with slopes $0$ and $1$, and which satisfy the following properties:
\begin{itemize}
\labelsep=7mm
 \item[{\rm (1)}] $0\le L_1(q)\le \cdots\le L_n(q)\le q$ for each $q\ge 0$;
 \smallskip
 \item[{\rm (2)}] $|M_k(q)-L_1(q)-\cdots-L_k(q)|\le c$ for each $k=1,\dots,n$ and
      each $q\ge 0$;
 \smallskip
 \item[{\rm (3)}] $M_n(q)=q$ for each $q\ge 0$;
 \smallskip
 \item[{\rm (4)}] if, for some integer $k$ with $1\le k <n$ and some $q>0$, the function
  $M_k$ changes slope from $1$ to $0$ at $q$, then $|L_{k+1}(q)-L_k(q)| \le 2c$.
\end{itemize}
Choose $c'>24n^3c$ and set
\[
 t_i=(1+2+\cdots+i)c'  \quad \text{for $i=0,1,\dots,n$.}
\]
Then there exists an $n$-system $\uR=(R_1,\dots,R_n)$ on $[0,\infty)$, whose
restriction to $[t_n,\infty)$ is rigid of mesh $c'$, such that
\begin{itemize}
\labelsep=7mm
 \item[{\rm (5)}] $\max_{1\le k\le n} |L_k(q)-R_k(q)|\le 4n^2c'$ \ for each $q\ge 0$;
 \medskip
 \item[{\rm (6)}] $\uR(t_i)=\big(0,\dots,0,c',2c',\dots,ic'\big)$ \ for each $i=0,1,\dots,n$.
\end{itemize}
\end{proposition}

\begin{proof}
Define $M_0=0$ and $P_k=M_k-M_{k-1}$ for $k=1,\dots,n$.  Put also
$\gamma=6c$.  By adapting the proof of \cite[Theorem 2.9]{R2015},
we find that $\uP=(P_1,\dots,P_n)\colon [0,\infty) \to \bR^n$ is an
$(n,\gamma)$-system in the sense of \cite[Definition 2.8]{R2015}, with
$|L_k(q)-P_k(q)|\le \gamma$ for each $q\ge 0$ and $k=1,\dots,n$.
Then, arguing as in the proof of \cite[Theorem 8.2]{R2015}, we obtain
a rigid $n$-system $\uR=(R_1,\dots,R_n)\colon[t_n,\infty)\to\bR^n$
of mesh $c'$ which satisfies condition (5) for each $q\ge t_n$.
In particular, $\uR(t_n)$ is a strictly increasing sequence of positive
integer multiples of $c'$ with sum $t_n$, and so $\uR(t_n)=(c',2c',\dots,nc')$.
From this it follows that $\uR$ extends uniquely to an $n$-system on
$[0,\infty)$ satisfying condition (6) (see the proof of
\cite[Theorem 8.1]{R2015}).  For each $k=1,\dots,n$
and each $q\ge 0$, we have $0\le L_k(q), R_k(q)\le q$, thus
$|L_k-R_k|$ is bounded above by $\max\{t_n,4n^2c'\}=4n^2c'$
on $[0,\infty)$.
\end{proof}

Note that, for $c=0$, the hypotheses of Proposition \ref{comb:prop}
amount to asking that the map $\uL:=(L_1,\dots,L_n)$ itself is an
$n$-system on $[0,\infty)$ (and that $M_k=L_1+\cdots+L_k$ for each
$k=1,\dots,n$).  In fact, this is how $n$-systems are defined in \cite[\S2.5]{R2015}
(where they are called $(n,0)$-systems).   From this, we infer the
following result of approximation.

\begin{cor}
\label{comb:cor}
Let $c'>0$, let $q_0=(n^2-n+1)c'/2$, and let $\uL=(L_1,\dots,L_n)$
be an $n$-system on $[0,\infty)$.  Then, there exists an $n$-system
$\uR=(R_1,\dots,R_n)$ on $[0,\infty)$ whose restriction to $[q_0,\infty)$ is rigid
of mesh $c'/2$, which satisfies
$\max_{1\le k\le n} |L_k(q)-R_k(q)|\le 4n^2c'$ for each
$q\ge 0$, and for which $R_1$ has slope $1$ on $[q_0,q_0+c'/2]$.
\end{cor}

\begin{proof}
The conditions (1)--(4) of Proposition \ref{comb:prop} are satisfied for the choice of
$c=0$ and of $M_k=L_1+\cdots+L_k$ for each $k=1,\dots,n$. So its conclusion
applies for the given $c'$.  Consider the resulting $n$-system
$\uR=(R_1,\dots,R_n)$ on $[0,\infty)$.  On $[t_{n-1},t_n]$, the union of the
graphs of $R_1,\dots,R_n$ consists of $n-1$ horizontal line segments of ordinates
$c',2c',\dots,(n-1)c'$ and one line segment of slope $1$ joining $(t_{n-1},0)$
to $(t_n,nc')$.  Since $q_0=t_{n-1}+c'/2$, we deduce that
$\uR(q_0)=(c'/2,c',2c',\dots,(n-1)c')$ and that $R_1$ has slope $1$ on $[q_0,q_0+c'/2]$.
Finally, since $\uR$ is rigid of mesh $c'$ on $[t_n,\infty)$, it is also rigid
of mesh $c'/2$ on $[q_0,\infty)$.
\end{proof}

This corollary will be useful when it comes to approximate an $n$-system
$\uL$ by the map $\uL_\uxi$ attached to a non-zero point $\uxi\in K_\pw^n$, because
it reduces the problem to approximating an $n$-system $\uR$ as in the
corollary.  The property that $R_1$ has slope $1$ to the right
of $q_0$ will simplify the argument.

%
%

\section{From points to $n$-systems}
\label{sec:points}

The goal of this section is to prove the first and last assertions of Theorem A.
To this end, we fix an integer $n\ge 2$, a place $\pw\in \MK$, and a non-zero
point $\uxi\in K_\pw^n$.  As $\uL_\uxi$ and $\uL^*_\uxi$ depend only on the
line $\langle\uxi\rangle_{K_\pw}$ spanned by $\uxi$ in $K_\pw^n$, we
assume, to simplify the computations, that
\[
 \norm{\uxi}_\pw=1.
\]
Using the general strategy of Schmidt and Summerer in \cite{SS2013}, we will
show that the components $L_1,\dots,L_n$ of $\uL_\uxi$ satisfy the hypotheses
of Proposition \ref{comb:prop} for some choice of functions $M_1,\dots,M_n$
and some constant $c=c(K,\pw,n)\ge 1$.  This will ensure the existence of
an $n$-system $\uP\colon[0,\infty)\to\bR^n$ for which the difference
$\uL_\uxi-\uP$ is bounded, and we will show that this is equivalent to
$\uL^*_\uxi-\uP^*$ being bounded.  The precise argument given below is
adapted from \cite[\S2]{R2015}.  In all estimates, the implicit constants
involved in the symbol $\asymp$ depend only on $K$, $\pw$ and $n$.
We also use the convention that a hat on an element of a sequence or product 
means that this element is omitted from the sequence or product.

For each $k\in\{1,\dots,n-1\}$, there is a unique bilinear map
\[
 \begin{array}{rcl}
  K_\pw^n\times \tbigwedge^kK_\pw^n &\longrightarrow &\tbigwedge^{k-1}K_\pw^n\\[2pt]
  (\uy,\uX) &\longmapsto &\uy\iprod\uX
 \end{array}
\]
called \emph{contraction} which satisfies
\begin{equation}
\label{points:eq:contraction}
 \uy\iprod(\ux_1\wedge\cdots\wedge\ux_k)
  = \sum_{i=1}^k (-1)^{i-1} (\uy\cdot\ux_i)
       \ux_1\wedge\dots\wedge\widehat{\ux_i\,}\wedge\cdots\wedge\ux_k
\end{equation}
for any $\uy,\ux_1,\dots,\ux_k\in K_\pw^n$.  For $k=1$, this is simply
the dot product $\uy\iprod\ux=\uy\cdot\ux$.  We use this to define
a map $\uL_\uxi^{(k)}$ as follows.

\begin{definition}
\label{points:def:L}
Let $k\in\{1,\dots,n-1\}$ and let $N=\binom{n}{k}$.  For each non-zero
$\uX\in\bigwedge^kK^n$, we set
\[
 D_\uxi(\uX)=\norm{\uxi\iprod\uX}_\pw^{d_\pw/d}\prod_{\pv\neq w}\norm{\uX}_\pv^{d_\pv/d}.
\]
As $\norm{\uxi}_\pw=1$, this agrees with the definition of section \ref{results:ssec:exp}
for $k=1$.  We also define a map $L_\uxi(\uX,\cdot)\colon[0,\infty)\to\bR$ by
\begin{equation}
\label{points:eq:LxiX}
 L_\uxi(\uX,q)=\max\big\{\log H(\uX),\,q+\log D_\uxi(\uX)\} \quad (q\ge 0).
\end{equation}
For each $j=1,\dots,N$ and $q\ge 0$, we denote by $L^{(k)}_{\uxi,j}(q)$ the
smallest real number $t\ge 0$ for which there exist at least $j$ linearly
independent elements $\uX$ of $\bigwedge^kK^n$ for which $L_\uxi(\uX,q)\le t$
or equivalently for which
\[
  H(\uX)\le e^t \et D_\uxi(\uX)\le e^{t-q}.
\]
Finally, we define $\uL_{\uxi}^{(k)}\colon[0,\infty)\to\bR^N$ by
$\uL_\uxi^{(k)}(q)=\big( L_{\uxi,1}^{(k)}(q),\dots,L_{\uxi,N}^{(k)}(q)\big)$ for each
$q\ge 0$.
\end{definition}

Since, for a non-zero $\uX\in\bigwedge^kK^n$, the numbers $H(\uX)$ and
$D_\uxi(\uX)$ depend only on the class of $\uX$ in projective space on
$\bigwedge^kK^n$, and since for each $B\ge 1$ there are finitely many classes
of height at most $B$, each number $L^{(k)}_{\uxi,j}(q)$ can indeed be defined
as a minimum.   For $k=1$, we recover $\uL^{(1)}_\uxi=\uL_\uxi$.   The first step
is to compare these maps with the minima of the following families of adelic convex
bodies.

\begin{definition}
\label{points:def:C}
Let $k$ and $N$ be as in Definition \ref{points:def:L}.  For each $q\ge 0$, we denote
by $\cC^{(k)}_\uxi(q)$ the convex body of $\bigwedge^kK_\bA^n$ which consists
of the points $\uX=(\uX_\pv)$ satisfying
\[
 \text{$\norm{\uX_\pv}_\pv\le 1$ \ for each $\pv\in \MK$ and \
  $\norm{\uxi\iprod\uX_\pw}_\pw\le e^{-qd/d_\pw}$.}
\]
We also set $\cC_\uxi(q)=\cC^{(1)}_\uxi(q)$.
\end{definition}

Thus $\cC_\uxi(q)$ consists of the points $(\ux_\pv)\in K_\bA^n$ satisfying
\[
 \text{$\norm{\ux_\pv}_\pv\le 1$ \ for each $\pv \in \MK$ \ and \
  $|\ux_\pw\cdot\uxi|_\pw\le e^{-qd/d_\pw}$.}
\]
Its volume is $\mu(\cC_\uxi(q))\asymp e^{-qd}$.  Applying Definition \ref{dilations:def},
we first obtain the following estimate.

\begin{lemma}
\label{points:lemma1}
Let $k\in\{1,\dots,n-1\}$, let $\uX\in \bigwedge^kK^n\setminus\{0\}$ and let $q\ge 0$.  
Then, we have 
\[
 \lambda(\uX,\cC^{(k)}_\uxi(q))
  \asymp \max\{H(\uX),e^qD_\uxi(\uX)\} = \exp(L_\uxi(\uX,q)).
\]
\end{lemma}

\begin{proof}
An id\`ele $\ua=(a_\pv)$ of $K$ of smallest module such that
$\uX\in\ua\cC^{(k)}_\uxi(q)$ has $|a_\pv|_\pv=\norm{\uX}_\pv$ for
each place $\pv \neq \pw$ and
$|a_\pw|_\pw = \max\{\norm{\uX}_\pw,\,r\norm{\uxi\iprod\uX}_\pw\}$
where $r$ is the smallest
element of the valuation group $|K_\pw^\mult|_\pw$ at $\pw$ with
$r\ge e^{qd/d_\pw}$. The estimate follows since $r\asymp e^{qd/d_\pw}$
and we have $\lambda(\uX,\cC^{(k)}_\uxi(q)) = \mua$ for such $\ua$.
\end{proof}

\begin{lemma}
\label{points:lemma2}
Let $k\in\{1,\dots,n-1\}$ and $N=\binom{n}{k}$.  For each $q\ge 0$, we have
\begin{equation}
\label{points:lemma2:eq1}
  0\le L^{(k)}_{\uxi,1}(q)\le \cdots\le L^{(k)}_{\uxi,N}(q)\le q
  \et
  \exp(L^{(k)}_{\uxi,j}(q))\asymp \lambda_j(\cC^{(k)}_\uxi(q))
  \quad (1\le j\le N).
\end{equation}
Moreover, the functions $L^{(k)}_{\uxi,j}$ are continuous and piecewise linear
with slopes $0$ and $1$ on $[0,\infty)$.  Finally, if $L^{(k)}_{\uxi,1}$ changes
slope from $1$ to $0$ at a point $q>0$, then $L^{(k)}_{\uxi,1}(q) =
L^{(k)}_{\uxi,2}(q)$.
\end{lemma}

\begin{proof}
For given $q\ge 0$ and $j\in\{1,\dots,N\}$, the number $L^{(k)}_{\uxi,j}(q)$ (resp.\
$\lambda_j^\bA(\cC^{(k)}_\uxi(q))$) is, by definition, the minimum of
\begin{equation}
\label{points:lemma2:eq2}
 \max_{\uX\in E} \exp(L_\uxi(\uX,q))
 \quad
 \Big(\ \text{resp.\ } \max_{\uX\in E} \lambda(\uX,\cC^{(k)}_\uxi(q))\ \Big)
\end{equation}
where $E$ runs through all sets of $j$ linearly independent elements of
$\bigwedge^kK^n$.  Taking for $E$ a set of $j$ products of the form
$\ue_{i_1}\wedge\cdots\wedge\ue_{i_k}$ with $1\le i_1<\cdots<i_k\le n$,
where $(\ue_1,\dots,\ue_n)$ is the canonical basis of $K^n$, we deduce that
$L^{(k)}_{\uxi,j}(q)\le q$ because those products $\ue$ have $D_\uxi(\ue)\le 1=H(\ue)$
as $\norm{\uxi}_\pw=1$.  This yields the first set of inequalities in
\eqref{points:lemma2:eq1}.   Using Lemma \ref{points:lemma1}, we also
deduce that $\exp(L^{(k)}_{\uxi,j}(q)) \asymp \lambda_j^\bA(\cC^{(k)}_\uxi(q))$.
The second set of estimates in \eqref{points:lemma2:eq1} then follows using
Proposition \ref{dilations:prop} with $\bigwedge^kK_\bA^n$ identified to $K_\bA^N$.

Fix $Q>0$.  For each $q\in[0,Q]$, we have $L^{(k)}_{\uxi,j}(q)\le q\le Q$.
Thus in computing $L^{(k)}_{\uxi,j}(q)$ on $[0,Q]$ in terms of the projective
invariants \eqref{points:lemma2:eq2}, it suffices to choose $E$ inside
a set $F$ of representatives in $\bigwedge^kK^n$ of points
of $\bP(\bigwedge^kK^n)$ of height at most $e^Q$.  Since $F$ is finite, we deduce
that $L^{(k)}_{\uxi,j}$ is continuous and piecewise linear
with slopes $0$ and $1$ on $[0,Q]$.  As $Q$ can be taken arbitrarily large,
this property extends to $[0,\infty)$.  Finally, if $L^{(k)}_{\uxi,1}$ changes slope
from $1$ to $0$ at a point $q>0$, there exist $\epsilon>0$ and
$\uX,\uY\in \bigwedge^kK^n$ such that
\[
 L^{(k)}_{\uxi,1}(t)
  =\begin{cases}
        L_\uxi(\uX,t)= t+\log D_\uxi(\uX) &\text{for $q-\epsilon\le t\le q$,}\\
        L_\uxi(\uY,t)=\log H(\uY) &\text{for $q\le t\le q+\epsilon$.}
    \end{cases}
\]
Thus $\uX,\uY$ are linearly independent and so $L^{(k)}_{\uxi,1}(q) =
L^{(k)}_{\uxi,2}(q)$.
\end{proof}

The next lemma compares the convex body $\cC_\uxi^{(k)}(q)$ with the
$k$-th compound of $\cC_\uxi(q)$.

\begin{lemma}
\label{points:lemma3}
Let $k$ and $N$ be as in Lemma \ref{points:lemma2}.  For each $q\ge 0$, we have
\begin{equation}
 \label{points:lemma3:eq1}
 \tbigwedge^k\cC_\uxi(q) \subseteq k\cC^{(k)}_\uxi(q)
 \et
 \cC^{(k)}_\uxi(q)  \subseteq N\tbigwedge^k\cC_\uxi(q).
\end{equation}
\end{lemma}

\begin{proof}
Fix a choice of $q\ge 0$ and, for simplicity, set
\[
  \cC:=\cC_\uxi(q)=\prod_\pv\cCv
  \et
  \cC^{(k)}:=\cC^{(k)}_\uxi(q)=\prod_\pv\cCv^{(k)}.
\]
Let $\pv \in \MK$ and let $\uX_\pv=\ux_1\wedge\cdots\wedge\ux_k$ with
$\ux_1,\dots,\ux_k\in\cCv$.  We find
\[
 \norm{\uX_\pv}_\pv
  \le \norm{\ux_1}_\pv\cdots\norm{\ux_k}_\pv \le 1,
\]
thus $\uX_\pv\in\cCv^{(k)}$ if $\pv \neq w$.  If $\pv=\pw$, formula
\eqref{points:eq:contraction} also yields $\norm{\uxi\iprod\uX_\pw}_\pw
\le k^\delta e^{-qd/d_\pw}$ where $\delta=1$ if $\pw\mid\infty$
and $\delta=0$ else, thus $\uX_\pw\in k^\delta\cC^{(k)}_\pw$.  This implies the first
inclusion in \eqref{points:lemma3:eq1}.

To prove the second inclusion, it suffices to show that each $\uX_\pv\in\cCv^{(k)}$
can be written as a sum of $N$ products of $k$ elements of $\cCv$.   For $\pv \neq \pw$,
this is immediate since $\cCv$ contains $\ue_1,\dots,\ue_n$ and since the products
$\ue_{i_1}\wedge\cdots\wedge\ue_{i_k}$ with $1\le i_1<\cdots<i_k\le n$ form
an orthonormal basis of $K_\pv^n$.  As $\norm{\uX_\pv}_\pv\le 1$, the $N$ coordinates
of $\uX_\pv$ in this basis have absolute values at most one, and we are done.
For $\pv=\pw$, we complete $\uu_1=\uxi$ into an orthonormal basis
$(\uu_1,\dots,\uu_n)$ of $K_\pw^n$  and we form its dual basis
$(\uu^*_1,\dots,\uu^*_n)$ with respect to the dot product.  By Lemma
\ref{metric:lemma:dual}, this new basis is also orthonormal.  Moreover we have
$\uxi\cdot\uu^*_1=1$ and $\uxi\cdot\uu^*_i=0$ for $i=2,\dots,n$.  Thus
$\cCw$ contains $\uu^*_2,\dots,\uu^*_n$ as well as $c\uu^*_1$ for any
$c\in K_\pw$ with $|c|_\pw\le e^{-qd/d_\pw}$.  Upon writing
\[
 \uX_\pw
  =\sum_{1\le i_1<\cdots<i_k\le n}c_{i_1,\dots,i_k}
     \uu^*_{i_1}\wedge\cdots\wedge\uu^*_{i_k},
\]
we find
\[
 \uxi\iprod\uX_\pw
  =\sum_{1<i_2<\cdots<i_k\le n}c_{1,i_2,\dots,i_k}
     \uu^*_{i_2}\wedge\cdots\wedge\uu^*_{i_k}.
\]
As $\norm{\uX_\pw}_\pw\le 1$ and $\norm{\uxi\iprod\uX_\pw}_\pw\le e^{-qd/d_\pw}$,
we deduce that $|c_{i_1,\dots,i_k}|_\pw$ is bounded above by $e^{-qd/d_\pw}$
if $i_1=1$, and by $1$ otherwise, so we are done.
\end{proof}

We can now prove the following part of Theorem A.

\begin{proposition}
\label{points:prop1}
There exists an $n$-system $\uP\colon[0,\infty)\to\bR^n$
such that $\norm{\uL_\uxi-\uP}$ is uniformly bounded by a constant
depending only on $K$, $\pw$ and $n$.
\end{proposition}

\begin{proof}
Set $M_k=L^{(k)}_{\uxi,1}$ for $k=1,\dots,n-1$ and define $M_n(q)=q$ for each
$q\ge 0$.   It suffices to show that these functions and the functions $L_k=L_{\uxi,k}$
satisfy all the hypotheses of Proposition \ref{comb:prop} for some constant
$c=c(K,\pw,n)$.  By Lemma \ref{points:lemma2}, we only have to verify
conditions (2) and (4) of that proposition.   Theorem \ref{adelic:thm:MBV} gives
\begin{equation}
\label{points:prop1:eq}
 \lambda_1(\cC_\uxi(q))\cdots\lambda_n(\cC_\uxi(q))
     \asymp \mu(\cC_\uxi(q))^{-1/d} \asymp e^q,
\end{equation}
while for each $k=1,\dots,n-1$, Theorem \ref{adelic:thm:Burger} provides
\begin{align*}
 &\lambda_1\big(\tbigwedge^k\cC_\uxi(q)\big)
 \asymp \lambda_1(\cC_\uxi(q))\cdots\lambda_k(\cC_\uxi(q)),\\
 &\lambda_2\big(\tbigwedge^k\cC_\uxi(q)\big)
 \asymp \lambda_1(\cC_\uxi(q))\cdots\lambda_{k-1}(\cC_\uxi(q))
   \lambda_{k+1}(\cC_\uxi(q)).
\end{align*}
The inclusions of Lemma \ref{points:lemma3} combined with
\eqref{points:lemma2:eq1} also imply that
\[
 \lambda_j\big(\tbigwedge^k\cC_\uxi(q)\big)
   \asymp \lambda_j\big(\cC^{(k)}_\uxi(q)\big)
   \asymp \exp(L^{(k)}_{\uxi,j}(q))
\]
for $j=1,\dots,\binom{n}{k}$.  Taking logarithms, we obtain the inequalities
of Proposition \ref{comb:prop} (2) as well as
\[
 |L^{(k)}_{\uxi,2}(q)-L_1(q)-\cdots-L_{k-1}(q)-L_{k+1}(q)|\le c
   \quad (1\le k<n, \ q\ge 0),
\]
for some constant $c=c(K,\pw,n)$.  Finally, if, for some $k<n$, the function
$M_k$ changes slope from $1$
to $0$ at a point $q>0$, Lemma \ref{points:lemma2} gives
$L^{(k)}_{\uxi,2}(q)=L^{(k)}_{\uxi,1}(q)=M_k(q)$.  Then comparing the
last estimate with that of Proposition \ref{comb:prop} (2), we obtain
$|L_{k+1}(q)-L_k(q)|  \le 2c$.  Thus condition (4) of
Proposition \ref{comb:prop} holds as well.
\end{proof}

To compare approximation to $\uL_\uxi$ by an $n$-system $\uP$ and
approximation to $\uL^*_\uxi$ by the dual map $\uP^*$ defined by
\eqref{results:P*}, we first note the following equality.

\begin{lemma}
\label{points:lemmaL}
We have $\uL^{(n-1)}_\uxi=\uL^*_\uxi$.
\end{lemma}

\begin{proof}
Consider the $K$-linear map $\varphi\colon K^n \to \bigwedge^{n-1}K^n$ given by
\[
 \varphi(\ux)=\ux\iprod(\ue_1\wedge\cdots\wedge\ue_n)
\]
for each $\ux\in K^n$.  Writing $\ux=(x_1,\dots,x_n)$, we find that
\[
 \varphi(\ux)
  =\sum_{i=1}^n (-1)^{i-1} x_i
    \ue_1\wedge\cdots\wedge\widehat{\,\ue_i\,}\wedge\cdots\wedge\ue_n.
\]
Thus $\varphi$ is an isomorphism and, for each place $\pv$ of $K$, it extends to
a $K_\pv$-linear isometry $\varphi_\pv\colon K_\pv^n \to \bigwedge^{n-1}K_\pv^n$
given by the same formulas.  Moreover, a short computation shows that
\[
 \norm{\uxi\iprod\varphi_\pw(\ux)}_\pw=\norm{\ux\wedge\uxi}_\pw
\]
for each $\ux\in K_\pw^n$.  Thus, for each non-zero $\ux\in K^n$, we have
\[
 H(\varphi(\ux))=H(\ux)
 \et
 D_\uxi(\varphi(\ux)) = D^*_\uxi(\ux),
\]
and the conclusion follows since $\varphi$ is an isomorphism.
\end{proof}

\begin{lemma}
\label{points:lemmaLc}
There is a constant $c^*=c^*(K,\pw,n)$ such that $|L^*_{\uxi,j}(q)+L_{\uxi,k}(q)-q|\le c^*$
for each $q\ge 0$ and each $j,k\in\{1,\dots,n\}$ with $j+k=n+1$.
\end{lemma}

\begin{proof}
For $q$, $j$ and $k$ as above, Theorem \ref{adelic:thm:Burger} combined
with \eqref{points:prop1:eq} provides
\begin{align*}
 &\lambda_j\big(\tbigwedge^{n-1}\cC_\uxi(q)\big)
 \asymp \lambda_1(\cC_\uxi(q))\cdots\widehat{\lambda_k(\cC_\uxi(q))}
     \cdots\lambda_{n}(\cC_\uxi(q))
 \asymp \frac{e^q}{\lambda_k(\cC_\uxi(q))}.
\end{align*}
Using Lemmas \ref{points:lemma3}, \ref{points:lemma2} and \ref{points:lemmaL}
in this order, we also find
\[
 \lambda_j\big(\tbigwedge^{n-1}\cC_\uxi(q)\big)
 \asymp \lambda_j(\cC^{(n-1)}_\uxi(q))
 \asymp \exp(L_{\uxi,j}^{(n-1)}(q))
 = \exp(L^*_{\uxi,j}(q)),
\]
while $\lambda_k(\cC_\uxi(q))\asymp \exp(L_{\uxi,k}(q))$.  The conclusion
follows by taking logarithms.
\end{proof}

We deduce the following complement to Proposition \ref{points:prop1}.

\begin{proposition}
\label{points:prop2}
Let $\uP=(P_1,\dots,P_n)$ be an $n$-system on $[0,\infty)$.  If one of the conditions
\eqref{results:thmA:eq} from Theorem A holds with a constant $c$, then the
other holds with $c$ replaced by $c+c^*\sqrt{n}$ where $c^*$ comes from Lemma
\ref{points:lemmaLc}.
\end{proposition}

\begin{proof}
With $q$, $j$ and $k$ as in Lemma \ref{points:lemmaLc}, the definition of $\uP^*$
in \eqref{results:P*} gives $P^*_j(q)+P_k(q)=q$.  Hence the inequality of the lemma
may be restated as
\[
 \big|(L^*_{\uxi,j}(q)-P^*_j(q))+(L_{\uxi,k}(q)-P_k(q))\big|\le c^*,
\]
and the result follows.
\end{proof}

We conclude this section with the following observation.

\begin{lemma}
\label{points:lemmaC*}
For $q\ge0$, define $\cC^*_\uxi(q)$ to be the set of points
$(\ux_\pv) \in K_\bA^n$ satisfying
\[
 \text{$\norm{\ux_\pv}_\pv\le 1$ \ for each $\pv \in \MK$ \ and \
  $\norm{\ux_\pw\wedge\uxi}_\pw\le e^{-qd/d_\pw}$.}
\]
Then $\lambda_j(\cC^*_\uxi(q))=\lambda_j(\cC_\uxi^{(n-1)}(q))
\asymp \exp(L^{(n-1)}_{\uxi,j}(q))=\exp(L^*_{\uxi,j}(q))$ for each
$j=1,\dots,n$.
\end{lemma}

\begin{proof}
Going back to the proof of Lemma \ref{points:lemmaL}, we find that the
$K_\bA$-linear isomorphism $\varphi_\bA\colon K_\bA^n \to
\bigwedge^{n-1}K_\bA^n$ which extends $\varphi$ maps $\cC^*_\uxi(q)$
to $\cC^{(n-1)}_\uxi(q)$ for each $q\ge 0$.  Thus these convex bodies
have the same minima.  The remaining estimates follow from
\eqref{points:lemma2:eq1} in Lemma \ref{points:lemma2}, and from
Lemma \ref{points:lemmaL}.
\end{proof}

%
%

\section{Construction of bases}
\label{sec:constr}

As in the preceding section, we assume $n\ge 2$, and we fix a place $\pw\in\MK$.
We also set
\[
 S=\MKinf\cup\{\pw\}
\]
and denote by $\cO_S=\cap_{\pv\notin S}(K\cap\cO_\pv)$ the ring of $S$-integers of $K$.
The goal of this section is to provide a general recursive construction of bases of
$\cO_S^n$ as an $\cO_S$-module.  In the next section, we will use it to complete
the proof of Theorem A.  The general strategy is similar to that of \cite[\S 5]{R2015}
but complicated by the fact that we need these bases to obey several properties
at each place of $S$.  In particular, we will need them to be almost orthogonal in
$K_\pv^n$ for each $\pv\in S\setminus\{\pw\}$, in the sense of section \ref{metric:ssec:ao}
for $L=K_\pv$.  We start by recalling two general results of approximation by
elements of $\cO_S$ within $\prod_{\pv\in S}K_\pv$.

The group of $S$-units of $K$ is the group $\cO_S^*$ of invertible elements of $\cO_S$.
It is well-known that its image under the logarithmic embedding, namely the set of
points $(\log|\varepsilon|_\pv)_{\pv\in S}$ with $\varepsilon\in\cO^*_S$,
forms a lattice within the hyperplane of $\bR^S$ of points $(x_\pv)_{\pv\in S}$ with
$\sum_{\pv\in S}d_\pv x_\pv=0$.  Thus, there is a constant $c_2=c_2(K,S)\ge 1$ with
the following property.

\begin{lemma}
\label{constr:lemma:units}
For any choice of positive real numbers $(r_\pv)_{\pv\in S}$ with
$\prod_{\pv\in S}r_\pv^{d_\pv}=1$, there exists an $S$-unit $\varepsilon\in\cO_S^*$
which satisfies $c_2^{-1}r_\pv\le |\varepsilon|_\pv\le c_2r_\pv$ for each $\pv \in S$.
\end{lemma}

From \cite[Theorem 3]{Ma1964} of Mahler, there is also a constant $c_3=c_3(K,S)\ge 1$
with the following property.

\begin{lemma}
\label{constr:lemma:integers}
For any $(a_\pv)_{\pv\in S} \in \prod_{\pv\in S}K_\pv$ and any family of positive real
numbers $(t_\pv)_{\pv\in S}$ with $\prod_{\pv\in S}t_\pv^{d_\pv/d}\ge c_3$, there
exists an $S$-integer $\alpha\in\cO_S$ which satisfies $|\alpha-a_\pv|_\pv\le t_\pv$
for each $\pv \in S$.
\end{lemma}

In fact, the result of Mahler shows this with a constant $c_3$ that depends only on $K$ when
each $t_\pv$ belongs to the valuation group of $K_\pv$.  As we do not
require this, our constant $c_3$ depends on $S$ as well, but in a weak form.

From now on, we fix a constant $C$ in the valuation group of $K_\pw$, with
\begin{equation}
\label{constr:eq0}
 C\ge n2^{n+1}(c_3c_4)^{d/d_\pw}
 \quad\text{where}\quad
 c_4=n2^n(2ec_2)^2.
\end{equation}
In agreement with the notation of section \ref{metric:ssec:ao} for $L=K_\pw$, we set
\[
 \delta=\begin{cases} 1 &\text{if $\pw\mid\infty$,}\\ 0 &\text{otherwise.}\end{cases}
\]
For the other places $\pv\in S\setminus\{\pw\}$, no special notation is needed since
they all are archimedean and so the results of that section apply to $L=K_\pv$ with
$\delta=1$.  For convenience, we also define
\begin{equation}
\label{constr:eq:Delta}
 \Delta=\{(a_1,\dots,a_n)\in\bZ^n\,;\, 0\le a_1\le \cdots\le a_n\}.
\end{equation}
We are interested in bases of $\cO_S^n$ with three kinds of properties.

\begin{definition}
\parindent=0pt
Let $\ux=(\ux_1,\dots,\ux_n)$ be a basis of $\cO_S^n$ over $\cO_S$.  We say
that
\begin{itemize}[label=\textbullet, labelindent=10pt, leftmargin=*]
\item $\ux$ is \emph{admissible} if, for any $\pv \in \Sw$,   it is almost
  orthogonal in $K_\pv^n$ and satisfies $1\le \norm{\ux_j}_\pv\le (2ec_2)^2$ \ for
  $j=1,\dots,n$;
\smallskip
\item
  $\ux$ has \emph{size}  $\ua=(a_1,\dots,a_n)\in\Delta$ \ if \
  $C^{a_j}\le \norm{\ux_j}_\pw\le (1+\delta)C^{a_j}$ \ for $j=1,\dots,n$;
\smallskip
\item
  $\ux$ has \emph{type}  $(k,\ell)$ for integers $1\le k<\ell\le n$ if
  \[
   \dist_\pw\big(\ux_\ell,\langle\ux_1,\dots,\widehat{\ux_k},\dots,\ux_{\ell-1}\rangle_{K_\pw})
   \ge 1-\frac{1}{2^{\ell-1}}.
  \]
\end{itemize}
\end{definition}

We start with two quick consequences.

\begin{lemma}
\label{constr:lemma0}
Suppose that $\ux=(\ux_1,\dots,\ux_n)$ is an admissible basis of $\cO_S^n$
over $\cO_S$ of size $\ua=(a_1,\dots,a_n)\in\Delta$.  Then, for
$j=1,\dots,\ell$, we have $C^{a_jd_\pw/d}\le H(\ux_j)\le c_5 C^{a_jd_\pw/d}$
where $c_5=(2ec_2)^2$.
\end{lemma}

\begin{proof}
Let $j\in\{1,\dots,n\}$.  For each place $\pv$ of $K$ not in $S$, the $n$-tuple
$\ux$ is a basis of $\cO_\pv^n$ over $\cO_\pv$, hence $\norm{\ux_j}_\pv=1$.
Thus we have $H(\ux_j)=\prod_{\pv\in S}\norm{\ux_j}_\pv^{d_\pv/d}$
and the conclusion follows.
\end{proof}

\begin{lemma}
\label{constr:lemmaV}
Let $k$, $\ell$, $m$ be integers with $1\le k<\ell\le m\le n$ and let
$\uy=(\uy_1,\dots,\uy_n)$ be an admissible basis of $\cO_S^n$ over $\cO_S$.
Suppose that the subsequences $(\uy_1,\dots,\widehat{\uy_\ell\,},\dots,\uy_m)$ and
$(\uy_1,\dots,\widehat{\uy_k},\dots,\uy_m)$ are both almost orthogonal
in $K_\pw^n$.  Then, the subspaces
\[
 V_1=\langle\uy_1,\dots,\widehat{\uy_\ell\,},\dots,\uy_m\rangle_{K_\pw}
 \et
 V_2=\langle\uy_1,\dots,\widehat{\uy_k},\dots,\uy_m\rangle_{K_\pw}
\]
that they span in $K_\pw^n$ satisfy
\[
 \dist_\pw(V_1,V_2)^{d_\pw/d}
 \le e^{4}\frac{H(\langle\uy_1,\dots,\uy_m\rangle_K)}{H(\uy_1)\cdots H(\uy_m)}.
\]
\end{lemma}

\begin{proof}
Using Lemma \ref{metric:lemmaV} with $L=K_\pw$, we find
\[
  \dist_\pw(V_1,V_2)^{d_\pw/d}
  \le \left( e^{4\delta}\frac{\norm{\uy_1\wedge\cdots\wedge\uy_m}_\pw}%
       {\norm{\uy_1}_\pw\cdots\norm{\uy_m}_\pw}\right)^{d_\pw/d}
   = A \frac{H(\langle\uy_1,\dots,\uy_m\rangle_K)}{H(\uy_1)\cdots H(\uy_m)},
\]
where
\[
  A = e^{4\delta d_\pw/d}
        \prod_{\pv\in S\setminus\{\pw\}}\left(\frac{\norm{\uy_1}_\pv\cdots\norm{\uy_m}_\pv}%
           {\norm{\uy_1\wedge\cdots\wedge\uy_m}_\pv}\right)^{d_\pv/d}.
\]
Since $\uy$ is an admissible basis of $\cO_S^n$, the sequence $(\uy_1,\dots,\uy_m)$
is almost orthogonal in $K_\pv^n$ for each place $\pv\in S$ other than $\pw$.  As such
a place $\pv$ is archimedean, Lemma \ref{metric:lemmaQ} shows that the corresponding
factor of $A$ is bounded above by $e^{2d_\pv/d}$.  This gives $A\le e^4$.
\end{proof}

The main result of this section is the following construction which, in essence,
generalizes \cite[Lemma 5.1]{R2015}.  The crucial novelty is the introduction
of an $S$-unit $\varepsilon$ in condition (3) below (in the context of \cite{R2015}
where $K=\bQ$ and $S=\{\infty\}$, it would simply be $\varepsilon=\pm 1$).

\begin{lemma}
\label{constr:lemma1}
Let $h$, $k$, $\ell$ be integers with
\[
 1 \le k < \ell \le n \et 1\le h\le \ell \le n.
\]
Suppose that elements $\ua=(a_1,\dots,a_n)$ and $\ub=(b_1,\dots,b_n)$
of $\Delta$ satisfy
\begin{itemize}
\labelsep=7mm
\item[{\rm (1)}] $b_\ell>a_\ell
 \et
 (b_1,\dots,\widehat{\,b_\ell\,},\dots,b_n)=(a_1,\dots,\widehat{a_h\,},\dots,a_n)$.
\end{itemize}
Suppose moreover that $\ux=(\ux_1,\dots,\ux_n)$ is an admissible basis
of $\cO_S^n$ over $\cO_S$ of size $\ua$.  Then there exists an admissible
basis $\uy=(\uy_1,\dots,\uy_n)$ of $\cO_S^n$ over $\cO_S$ of size $\ub$
and type $(k,\ell)$ such that
\begin{itemize}
\labelsep=7mm
\item[{\rm (2)}] $(\uy_1,\dots,\widehat{\uy_\ell\,},\dots,\uy_n)
                  = (\ux_1,\dots,\widehat{\ux_h},\dots,\ux_n)$,
\medskip
\item[{\rm (3)}] $\uy_\ell \in
   \varepsilon\ux_h + \langle\ux_1,\dots,\widehat{\ux_h},\dots,\ux_\ell\rangle_{\cO_S}$
   for some $\varepsilon\in\cO_S^*$.
\end{itemize}
\end{lemma}

\begin{proof}
We use (2) as a definition of $\uy_1,\dots,\widehat{\uy_\ell\,},\dots,\uy_n$.
Then $\uy=(\uy_1,\dots,\uy_n)$ is a basis of $\cO_S^n$ over $\cO_S$ for any
choice of $\uy_\ell$ satisfying (3).

We set $r_\pv=2e^2c_2\norm{\ux_h}_\pv^{-1}$ for each $\pv \in \Sw$,
and define $r_\pw$ by the condition $\prod_{\pv\in S}r_\pv^{d_\pv}=1$.  Since $\ux$
is admissible, we have $r_\pv\ge (2c_2)^{-1}$ for $\pv \neq w$, so $r_\pw^{d_\pw}\le
(2c_2)^{d}$.  Then Lemma \ref{constr:lemma:units} provides an $S$-unit
$\varepsilon\in\cO_S^*$ with
\begin{align}
&2e^2 \le \norm{\varepsilon\ux_h}_\pv \le 2e^2c_2^2
     \quad\text{for each $\pv \in \Sw$,}
\label{constr:lemma1:eq4}\\
&|\varepsilon|_\pw^{d_\pw}\le (c_2r_\pw)^{d_\pw}\le (2c_2)^{2d}.
\label{constr:lemma1:eq5}
\end{align}
For each $\pv\in S$, we define
\[
\begin{aligned}
  U_\pv&=\langle\ux_1,\dots,\ux_\ell\rangle_{K_\pv},\\
  V_\pv&=\langle\ux_1,\dots,\widehat{\ux_h},\dots,\ux_\ell\rangle_{K_\pv}
           =\langle\uy_1,\dots,\uy_{\ell-1}\rangle_{K_\pv},
\end{aligned}
\]
and we choose a unit vector $\uu_\pv\in U_\pv$ such that
\[
 U_\pv = \langle\uu_\pv\rangle_{K_\pv} \ptop V_\pv.
\]
If $\pv\neq\pw$, we write
\begin{equation}
\label{constr:lemma1:eq6}
 \varepsilon\ux_h = c_\pv\uu_\pv + \sum_{j=1}^{\ell-1} c_{\pv,j}\uy_j
\end{equation}
with coefficients $c_\pv$ and $c_{\pv,j}$ in $K_\pv$.  For $\pv=\pw$, we also define
\[
 W_\pw=\langle\uy_1,\dots,\widehat{\uy_k},\dots,\uy_{\ell-1}\rangle_{K_\pw},
\]
and choose a unit vector $\uv_\pw\in V_\pw$ such that
\[
 V_\pw = \langle\uv_\pw\rangle_{K_\pw} \ptop W_\pw.
\]
This provides a decomposition $U_\pw= \langle\uu_\pw\rangle_{K_\pw} \ptop
\langle\uv_\pw\rangle_{K_\pw} \ptop W_\pw$.  We choose $B\in K_\pw$ with
\begin{equation}
\label{constr:lemma1:eq7}
 |B|_\pw=(1+\delta/2)C^{b_\ell}.
\end{equation}
This is possible because if $\pw\nmid\infty$, then $\delta=0$ and $C$ belongs
to the valuation group of $K_\pw$.  We then write
\begin{equation}
\label{constr:lemma1:eq8}
 \varepsilon\ux_h
  = c_\pw\uu_\pw + B\uv_\pw + \sum_{j=1}^{\ell-1} c_{\pw,j}\uy_j
\end{equation}
with $c_\pw$ and $c_{\pw,j}$ in $K_\pw$.  The approximation lemma
\ref{constr:lemma:integers} provides, for each $j=1,\dots,\ell-1$, an $S$-integer
$\alpha_j\in\cO_S$ such that
\begin{equation}
\label{constr:lemma1:eq9}
\begin{aligned}
 &|\alpha_j-c_{\pv,j}|_\pv \le c_4^{-1}
    \quad \text{for all $\pv\in \Sw$,}\\
 &|\alpha_j-c_{\pw,j}|_\pw \le (c_3c_4)^{d/d_\pw}
\end{aligned}
\end{equation}
for the constant $c_4=n2^n(2ec_2)^2$ defined in \eqref{constr:eq0}.  Then
the point
\begin{equation}
\label{constr:lemma1:eq10}
 \uy_\ell = \varepsilon\ux_h - \sum_{j=1}^{\ell-1} \alpha_j\uy_j
\end{equation}
fulfills condition (3) and so $\uy=(\uy_1,\dots,\uy_n)$ is an $\cO_S$-basis
of $\cO_S^n$.

\medskip
$\mathbf{1}^\circ$
To show that $\uy$ is admissible, we fix a place $\pv\in\Sw \subseteq \MKinf$.
Since $\ux$ is admissible, we obtain directly $1\le \norm{\uy_j}_\pv \le (2ec_2)^2$
for each $j\neq\ell$ because of equality (2).  As $\ux$ is almost orthogonal in
$K_\pv^n$, its subsequence
\[
 (\ux_1,\dots,\widehat{\ux_h},\dots,\ux_\ell) = (\uy_1,\dots,\uy_{\ell-1})
\]
is also almost orthogonal.  Moreover, for each integer $j$ with $\ell+1\le j\le n$,
we have $\ux_j=\uy_j$ and $\langle\ux_1,\dots,\ux_{j-1}\rangle_{K_\pv}
=\langle\uy_1,\dots,\uy_{j-1}\rangle_{K_\pv}$, thus
\[
 \dist_\pv\big(\uy_j,\ \langle\uy_1,\dots,\uy_{j-1}\rangle_{K_\pv}\big)
 = \dist_\pv\big(\ux_j,\ \langle\ux_1,\dots,\ux_{j-1}\rangle_{K_\pv}\big)
 \ge 1-(1/2)^{j-1}.
\]
Since $\langle\uy_1,\dots,\uy_{\ell-1}\rangle_{K_\pv} = V_\pv$, it only remains
to show that
\begin{equation}
\label{constr:lemma1:eq11}
 1\le \norm{\uy_\ell}_\pv \le (2ec_2)^2
 \et
 \dist_\pv(\uy_\ell, V_\pv)\ge 1-(1/2)^{\ell-1}.
\end{equation}

To this end, we first note that, since $(\ux_1,\dots,\ux_n)$ is almost orthogonal
in $K_\pv^n$, Lemmas \ref{metric:lemmaCbis} and \ref{metric:lemmaQ} yield
\[
 1 \ge \dist_\pv(\ux_h, V_\pv)
  = \frac{\norm{\ux_1\wedge\cdots\wedge\ux_\ell}_\pv}%
     {\norm{\ux_h}_\pv\norm{\ux_1\wedge\cdots\wedge\widehat{\ux_h}%
      \wedge\cdots\wedge\ux_\ell}_\pv}
  \ge \frac{\norm{\ux_1\wedge\cdots\wedge\ux_\ell}_\pv}%
     {\norm{\ux_1}_\pv\cdots\norm{\ux_\ell}_\pv}
  \ge e^{-2}.
\]
On the other hand, Lemma \ref{metric:lemmaC}
applied to the decomposition of $\epsilon\ux_h$ in \eqref{constr:lemma1:eq6} gives
\[
 \dist_\pv(\ux_h, V_\pv) = \dist_\pv(\varepsilon\ux_h, V_\pv)
  =\frac{|c_\pv|_\pv}{\norm{\varepsilon\ux_h}_\pv}.
\]
Using the estimates of \eqref{constr:lemma1:eq4} for $\norm{\varepsilon\ux_h}_\pv$,
we deduce that
\begin{equation}
\label{constr:lemma1:eq12}
 2 \le |c_\pv|_\pv \le 2e^2c_2^2.
\end{equation}
Combining \eqref{constr:lemma1:eq6} and  \eqref{constr:lemma1:eq10}, we
obtain
\[
 \uy_\ell = c_\pv\uu_\pv + \sum_{j=1}^{\ell-1}(c_{\pv,j}-\alpha_j)\uy_j.
\]
Using \eqref{constr:lemma1:eq9}, this decomposition of $\uy_\ell$ implies
\[
 \norm{\uy_\ell-c_\pv\uu_\pv}_\pv
   \le \sum_{j=1}^{\ell-1}c_4^{-1}\norm{\uy_j}_\pv
   \le nc_4^{-1}(2ec_2)^2 = 2^{-n},
\]
and thus $1\le \norm{\uy_\ell}_\pv \le (2ec_2)^2$ by \eqref{constr:lemma1:eq12}.
Finally Lemma \ref{metric:lemmaC} gives
\[
 \dist_\pv(\uy_\ell, V_\pv)
  = \frac{|c_\pv|_\pv}{\norm{\uy_\ell}_\pv}
  \ge \frac{|c_\pv|_\pv}{|c_\pv|_\pv+2^{-n}}
  \ge \frac{2}{2+2^{-n}}
  \ge 1-\frac{1}{2^{n+1}},
\]
which completes the proof of \eqref{constr:lemma1:eq11}. Thus $\uy$ is admissible.

\medskip
$\mathbf{2}^\circ$
We now show that $\uy$ has size $\ub$.  Since $\ux$ has size $\ua$, relations
(1) and (2) reduce the problem to showing that
\begin{equation}
\label{constr:lemma1:eq13}
 C^{b_\ell} \le \norm{\uy_\ell}_\pw \le (1+\delta)C^{b_\ell}.
\end{equation}

To prove this, we first combine \eqref{constr:lemma1:eq8}
and \eqref{constr:lemma1:eq10} to obtain
\[
 \uy_\ell = c_\pw\uu_\pw + B\uv_\pw + \sum_{j=1}^{\ell-1}(c_{\pw,j}-\alpha_j)\uy_j,
\]
and note that the decomposition of $\epsilon\ux_h$ in \eqref{constr:lemma1:eq8} implies
\[
 |c_\pw|_\pw = \norm{c_\pw\uu_\pw}_\pw
   \le \norm{\varepsilon\ux_h}_\pw \le (2c_2)^{2d/d_\pw}\norm{\ux_h}_\pw,
\]
where the last estimation comes from \eqref{constr:lemma1:eq5}.  Using
\eqref{constr:lemma1:eq9}, we deduce that
\[
 \norm{\uy_\ell-B\uv_\pw}_\pw
   \le |c_\pw|_\pw  + \sum_{j=1}^{\ell-1}(c_3c_4)^{d/d_\pw}\norm{\uy_j}_\pw
   \le (c_3c_4)^{d/d_\pw}\sum_{j=1}^{\ell} \norm{\ux_j}_\pw.
\]
We also note that $\norm{\ux_j}_\pw \le 2C^{a_j} \le 2C^{b_\ell-1}$ for
$j=1,\dots,\ell$ because $\ux$ has size $\ua$, and hypothesis (1) gives
$a_1\le\cdots\le a_\ell<b_\ell$.  Using hypothesis \eqref{constr:eq0} on $C$, we
conclude that
\begin{equation}
\label{constr:lemma1:eq14}
 \norm{\uy_\ell-B\uv_\pw}_\pw
   \le 2n(c_3c_4)^{d/d_\pw}C^{b_\ell-1} \le 2^{-n}C^{b_\ell}.
\end{equation}
Using the value  for $|B|_\pw$ in  \eqref{constr:lemma1:eq7}, this
yields \eqref{constr:lemma1:eq13}.  So $\uy$ has size $\ub$.

\medskip
$\mathbf{3}^\circ$
It remains to show that $\uy$ has type $(k,\ell)$.

By \eqref{constr:lemma1:eq14}, we have a decomposition
$\uy_\ell=B\uv_\pw+\uz$ with $\norm{\uz}_\pw\le 2^{-n}|B|_\pw$.  Thus,
\begin{align*}
 \dist_\pw(\uy_\ell,W_\pw)
  &\ge \frac{|B|_\pw-\norm{\uz}_\pw}{\norm{\uy_\ell}_\pw}
   \ge \frac{1-2^{-n}}{1+2^{-n}}
   \ge 1-\frac{1}{2^{n-1}}
   &&\text{if $\pw\mid\infty$,}\\[5pt]
 \dist_\pw(\uy_\ell,W_\pw)
  &\ge \frac{|B|_\pw}{\norm{\uy_\ell}_\pw} =1
   &&\text{otherwise}.
\end{align*}
In both cases, this yields $\dist_\pw(\uy_\ell,W_\pw)\ge 1-\delta/2^{\ell-1}$.
So $\uy$ has type $(k,\ell)$.
\end{proof}

We will also need the following complementary result.

\begin{lemma}
\label{constr:lemma1bis}
With the same hypotheses and notation, let $m$ be an integer with $\ell\le m\le n$.
\begin{itemize}[labelindent=15pt, leftmargin=*]
\item[{\rm (i)}] We have $\langle\ux_1,\dots,\ux_m\rangle_K=
  \langle\uy_1,\dots,\uy_m\rangle_K$.  If $m<n$, we also have $\ux_{m+1}=
  \uy_{m+1}$.
\item[{\rm (ii)}] If $(\ux_1,\dots,\widehat{\ux_h},\dots,\ux_\ell)$ is almost
  orthogonal in $K_\pw^n$, then $(\uy_1,\dots,\widehat{\uy_k},\dots,\uy_\ell)$
  is as well almost orthogonal in $K_\pw^n$.
\item[{\rm (iii)}] If both $(\ux_1,\dots,\widehat{\ux_h},\dots,\ux_m)$ and
  $(\uy_1,\dots,\widehat{\uy_k},\dots,\uy_m)$ are almost orthogonal
  in $K_\pw^n$, then the subspaces of $K_\pw^n$ that they span,
  \[
  V_1=\langle\ux_1,\dots,\widehat{\ux_h},\dots,\ux_m\rangle_{K_\pw}
  \et
  V_2=\langle\uy_1,\dots,\widehat{\uy_k},\dots,\uy_m\rangle_{K_\pw}
  \]
  satisfy
  \begin{equation}
  \label{constr:lemma1:eq15}
  \dist_\pw(V_1,V_2)^{d_\pw/d}
  \le e^{4}\frac{H(\langle\uy_1,\dots,\uy_m\rangle_K)}{H(\uy_1)\cdots H(\uy_m)}
  \le e^4\frac{c_5}{C^{d_\pw/d}}
       \frac{H(\langle\ux_1,\dots,\ux_m\rangle_K)}{H(\ux_1)\cdots H(\ux_m)}.
  \end{equation}
\end{itemize}
\end{lemma}

When (iii) holds with $m<n$, estimates \eqref{constr:lemma1:eq15}
together with Corollary \ref{metric:cor:lemmaD} allow us to connect the distances from
$\ux_{m+1}=\uy_{m+1}$ to $V_1$ and $V_2$ in terms of heights only.

\begin{proof}
Part (i) follows immediately from conditions (2) and (3) of Lemma
\ref{constr:lemma1}.

To prove (ii), we note that $(\ux_1,\dots,\widehat{\ux_h},\dots,\ux_\ell)
=(\uy_1,\dots,\uy_{\ell-1})$.  Thus, if this sequence is almost orthogonal in
$K_\pw^n$, so is its subsequence $(\uy_1,\dots,\widehat{\uy_k},\dots,\uy_{\ell-1})$.
Since $\uy$ is of type $(k,\ell)$, we then conclude that
$(\uy_1,\dots,\widehat{\uy_k},\dots,\uy_{\ell})$ is almost orthogonal in
$K_\pw^n$.

Under the hypotheses of (iii), the first inequality in \eqref{constr:lemma1:eq15}
follows from Lemma \ref{constr:lemmaV}
because $(\ux_1,\dots,\widehat{\ux_h},\dots,\ux_m)$ coincides with
$(\uy_1,\dots,\widehat{\uy_\ell},\dots,\uy_m)$.  To prove the second inequality,
we use the fact that $\ux$ and $\uy$ have respective sizes $\ua$ and $\ub$ with
$a_h\le a_\ell<b_\ell$.  By Lemma \ref{constr:lemma0}, we obtain
\[
 \frac{H(\ux_1)\cdots H(\ux_m)}{H(\uy_1)\cdots H(\uy_m)}
  = \frac{H(\ux_h)}{H(\uy_\ell)}
  \le c_5 C^{(a_h-b_\ell)d_\pw/d}
  \le \frac{c_5}{C^{d_\pw/d}}.
\]
The required inequality follows since $H(\langle\uy_1,\dots,\uy_m\rangle_K)
=H(\langle\ux_1,\dots,\ux_m\rangle_K)$ by (i).
\end{proof}

We conclude with the following existence result which for us replaces
\cite[Lemma 5.2]{R2015}.

\begin{lemma}
\label{constr:lemma2}
Let $\ua=(a_1,\dots,a_n)\in\bZ^n$ with $0\le a_1<\cdots<a_n$.
There exists an admissible basis $\ux=(\ux_1,\dots,\ux_n)$ of $\cO_S^n$ over $\cO_S$
of size $\ua$ and type $(1,n)$ such that $(\ux_1,\dots,\ux_{n-1})$ is almost orthogonal
in $K_\pw^n$.
\end{lemma}

\begin{proof}
Starting from the canonical basis $(\ue_1,\dots,\ue_n)$ of $K^n$, Lemma
\ref{constr:lemma1} applied recursively $n$ times with $h=k=1$ and $\ell=n$
provides points $\ux_1,\dots,\ux_n$ of $\cO_S^n$ such that, for each
$j=0,\dots,n$, the $n$-tuple $(\ue_{j+1},\dots,\ue_n,\ux_1,\dots,\ux_j)$
is an admissible basis of $\cO_S^n$ of size $(0,\dots,0,a_1,\dots,a_j)$
and type $(1,n)$.  For each $j$ with $2\le j\le n-1$, we find
\[
 \dist_\pw\big(\ux_j,\langle\ux_1,\dots,\ux_{j-1}\rangle_{K_\pw}\big)
  \ge \dist_\pw\big(\ux_j,\langle\ue_{j+2},\dots,\ue_n,\ux_1,\dots,
            \ux_{j-1}\rangle_{K_\pw}\big)
  \ge 1-1/2^{n-1}.
\]
Thus $(\ux_1,\dots,\ux_{n-1})$ is almost orthogonal in $K_\pw^n$.
\end{proof}

%
%

\section{From $n$-systems to points}
\label{sec:converse}

Let $\uL\colon[0,\infty)\to\bR^n$ be an arbitrary $n$-system.  To complete the
proof of Theorem A, it remains to show the existence of a non-zero point
$\uxi\in K_\pw^n$ for which $\norm{\uL-\uL_\uxi}$ is bounded above by a constant
depending only on $K$, $\pw$ and $n$.  To this end, consider the $n$-system
$\uR=(R_1,\dots,R_n)$ provided by Corollary \ref{comb:cor} for the choice
of $c'=2c$, where
\[
 c=\frac{d_\pw}{d}\log(C)
\]
for the constant $C=C(K,\pw,n)$ of the preceding section, satisfying \eqref{constr:eq0}.
Since $\norm{\uL-\uR}$ is bounded above by a constant depending only on $K$, $n$
and $\pw$, it suffices to construct a non-zero point $\uxi\in K_\pw^n$ for which
$\norm{\uR-\uL_\uxi}$ is also bounded above by such a constant.

Let $q_0=(n^2-n+1)c$, so that the restriction of $\uR$ to $[q_0,\infty)$ is rigid
of mesh $c$.  We denote by $(q_i)_{0\le i<s}$ the finite or infinite sequence
of switch points of $\uR$ on that interval, with cardinality $s\in\{1,2,\dots\}\cup\{\infty\}$.
For each integer $i$ with $0\le i<s$, we set
\[
 \ua^{(i)}=c^{-1}\uR(q_i)=(a_1^{(i)},\dots,a_n^{(i)})\in \Delta,
\]
where $\Delta\subset\bZ^n$ is defined by \eqref{constr:eq:Delta}.  We have
$R_j(q_i)=ca_j^{(i)}$ for $j=1,\dots,n$ and
\begin{equation}
\label{converse:eq:qi}
 q_i = ca_1^{(i)}+\cdots+ca_n^{(i)} \quad (0\le i<s).
\end{equation}
We also denote by $k_i$ the index $j$ for which the right derivative of $R_j$ at
$q_i$ is $1$ and, when $i>0$, we denote by $\ell_i$ the index $j$ for which
the left derivative of $R_j$ at $q_i$ is $1$.  By the choice of $\uR$, we have
$k_0=1$.  Finally, we set $\ell_0=n$.  Then, for each integer $i$ with $1\le i<s$,
we have
\begin{itemize}
\item[{\rm (P1)}] $1=k_0 < \ell_0=n$ \ and \ $1\le k_i<\ell_i\le n$,\\
\item[{\rm (P2)}] $\ell_i\ge k_{i-1}$ \ and \ $a_{\ell_i}^{(i)}>a_{\ell_i}^{(i-1)}$,\\
\item[{\rm (P3)}] $\big(a_1^{(i)},\dots,\widehat{a_{\ell_i}^{(i)}},\dots,a_n^{(i)}\big)
  = \big(a_1^{(i-1)},\dots,\widehat{a_{k_{i-1}}^{(i-1)}},\dots,a_n^{(i-1)}\big)$.
\end{itemize}

From these data, it is a simple matter to reconstruct the function $\uR$.  Let
\[
 \Phi\colon \bR^n\to
  \{(x_1,\dots,x_n)\in\bR^n\,;\,x_1\le x_2\le \cdots\le x_n\}
\]
denote the continuous function which reorders the coordinates of a point as
a monotonically increasing sequence, and set $q_s=\infty$ if $s<\infty$.  Then,
for each $q\in [q_i,q_{i+1})$, we have
\begin{equation}
\label{converse:eq:R}
 \uR(q)=\Phi(\uR(q_i)+(q-q_i)\ue_{k_i})
          =\Phi(R_1(q_i),\dots,R_{k_i}(q_i)+q-q_i,\dots,R_n(q_i))
\end{equation}
The formula also extends to $q=q_{i+1}$ if $i+1<s$.

We first apply the results of the preceding section to construct a specific
sequence of bases of $\cO_S^n$.  Its relevance to our problem will become
clear in the corollaries that we derive afterwards.

\begin{proposition}
\label{converse:prop}
There exists a sequence $(\ux^{(i)})_{0\le i<s}$ of bases of $\cO_S^n$ over
$\cO_S$ such that, for each integer $i$ with $0\le i<s$, the basis $\ux^{(i)}=
(\ux_1^{(i)},\dots,\ux_n^{(i)})$ is admissible of size $\ua^{(i)}$ and type
$(k_i,\ell_i)$ with the additional property that, when $i\ge 1$,
\begin{itemize}
\labelsep=7mm
\item[{\rm (1)}]
  $\big(\ux_1^{(i)},\dots,\widehat{\ux_{\ell_i}^{(i)}},\dots,\ux_n^{(i)}\big)
    = \big(\ux_1^{(i-1)},\dots,\widehat{\ux_{k_{i-1}}^{(i-1)}},\dots,\ux_n^{(i-1)}\big)$,
\medskip
\item[{\rm (2)}]
  $\ux_{\ell_i}^{(i)} \in
   \varepsilon_i\ux_{k_{i-1}}^{(i-1)}
   + \langle\ux_1^{(i-1)},\dots,\widehat{\ux_{k_{i-1}}^{(i-1)}},\dots,
         \ux_{\ell_i}^{(i-1)}\rangle_{\cO_S}$
   for some $S$-unit $\varepsilon_i\in\cO_S^*$.
\end{itemize}
We may further require that the sequence
$\wx^{(-1)}:=(\ux_1^{(0)},\dots,\ux_{n-1}^{(0)})$ is almost orthogonal in
$K_\pw^n$.  Then,  for each integer $i$ with $0\le i<s$, the sequence
$\wx^{(i)}:=(\ux_1^{(i)},\dots,\widehat{\ux_{k_i}^{(i)}},\dots,\ux_n^{(i)})$ is
also almost orthogonal in $K_\pw^n$.  Finally, for each $i$ with $-1\le i<s$, choose
a unit vector $\uu_i$ of $K_\pw^n$ which is orthogonal to each vector of\/ $\wx^{(i)}$
with respect to the dot product.  Then we further have
\begin{equation}
\label{converse:prop:eq0}
 \dist_\pw(\uu_i,\uu_j) \le 2^\delta\exp((4-q_{i+1})d/d_\pw)
 \quad (-1\le i<j<s).
\end{equation}
\end{proposition}

\begin{proof}
Lemma \ref{constr:lemma1} provides recursively such a sequence of bases
starting from any admissible basis $\ux^{(0)}=(\ux_1^{(0)},\dots,\ux_n^{(0)})$
of size $\ua^{(0)}$ and type $(k_0,\ell_0)$.  To build $\ux^{(i)}$ from $\ux^{(i-1)}$
for an integer $i$ with $1\le i<s$, we apply this lemma with $h=k_{i-1}$,
$(k,\ell)=(k_i,\ell_i)$, $\ua=\ua^{(i-1)}$ and $\ub=\ua^{(i)}$.  The hypotheses
of the lemma are fulfilled by virtue of conditions (P1)--(P3).

By Lemma \ref{constr:lemma2}, we may choose the initial basis $\ux^{(0)}$ so that
$\wx^{(-1)}$ is almost orthogonal in $K_\pw^n$.
Assuming this, we now prove by induction that $\wx^{(i)}$ is almost
orthogonal in $K_\pw^n$ for each $i$ with $0\le i<s$.

We first note that $\wx^{(0)}=(\ux_2^{(0)},\dots, \ux_n^{(0)})$ is almost
orthogonal in $K_\pw^n$ because $\ux^{(0)}$ has type $(1,n)$ and the sequence
$(\ux_2^{(0)},\dots, \ux_{n-1}^{(0)})$ is almost orthogonal in $K_\pw^n$, as a
subsequence of the almost orthogonal sequence $\wx^{(-1)}$.

Suppose now that $\wx^{(i)}$ is almost orthogonal in $K_\pw^n$ for each
$i=0,\dots,t-1$ where $t$ is an integer with $1\le t <s$.  To complete the induction step,
we will show, by induction
on $m$, that $(\ux_1^{(t)},\dots,\widehat{\ux_{k_t}^{(t)}},\dots,\ux_m^{(t)})$
is almost orthogonal in $K_\pw^n$ for each $m=\ell_t,\dots,n$.  For $m=\ell_t$,
this follows from Lemma \ref{constr:lemma1bis} (ii) since
$(\ux_1^{(t-1)},\dots,\widehat{\ux_{k_{t-1}}^{(t-1)}},\dots,\ux_{\ell_t}^{(t-1)})$
is almost orthogonal in $K_\pw^n$.  If $\ell_t=n$, we are done. Otherwise,
let $m$ be an integer with $\ell_t\le m<n$ for which
$(\ux_1^{(t)},\dots,\widehat{\ux_{k_t}^{(t)}},\dots,\ux_m^{(t)})$  is almost
orthogonal in $K_\pw^n$.   Since $\ell_0=n>m$, there is a largest integer
$r$ with $0\le r<t$ such that $\ell_r>m$.  This means that $\ell_{r+1},\dots,\ell_t\le m$
and so $k_r,\dots,k_t\le m$ by (P1) and (P2).  Moreover, we have
$\ux_{m+1}^{(r)}=\cdots=\ux_{m+1}^{(t)}$ by Lemma \ref{constr:lemma1bis} (i).
Define
\begin{equation}
\label{converse:prop:eq:UV}
 U^{(i)}=\langle \ux_1^{(i)},\dots,\ux_m^{(i)} \rangle_{K_\pw}
 \et
 V^{(i)}=\langle \ux_1^{(i)},\dots,\widehat{\ux_{k_i}^{(i)}},\dots,\ux_m^{(i)} \rangle_{K_\pw}
\end{equation}
for $i=r,\dots,t$.  We claim that
\begin{equation}
\label{converse:prop:eq:claim}
 \dist_\pw(\ux_{m+1}^{(r)}, V^{(r)}) \ge 1-\frac{\delta}{2^m}
 \et
 \dist_\pw(V^{(r)},V^{(t)}) \le \frac{1}{2^m}.
\end{equation}
If we take this for granted, then Corollary \ref{metric:cor:lemmaD} gives
$\dist_\pw(\ux_{m+1}^{(t)}, V^{(t)}) \ge 1-\delta/2^{m-1}$
which is exactly what we need to complete the induction on $m$, and thus to complete
the main induction as well.

If $r\ge 1$ and $m+1<\ell_r$, the $(m+1)$-tuple
$(\ux_1^{(r)},\dots,\ux_{m+1}^{(r)})$ is almost orthogonal
in $K_\pw^n$ as a subsequence of
\[
 \big(\ux_1^{(r)},\dots,\widehat{\ux_{\ell_r}^{(r)}},\dots,\ux_n^{(r)}\big)
    = \big(\ux_1^{(r-1)},\dots,\widehat{\ux_{k_{r-1}}^{(r-1)}},\dots,\ux_n^{(r-1)}\big).
\]
If $r=0$ and $m+1<\ell_0=n$, it is also almost orthogonal in $K_\pw^n$, as
a subsequence of $\wx^{(-1)}$.   So, independently of $r$, if $m+1<\ell_r$, we obtain
\[
 \dist_\pw(\ux_{m+1}^{(r)}, V^{(r)})
 \ge \dist_\pw(\ux_{m+1}^{(r)}, U^{(r)}) \ge 1-\frac{\delta}{2^m},
\]
which gives the first inequality in \eqref{converse:prop:eq:claim}.  If $m+1=\ell_r$,
the latter inequality holds by construction, since $\ux^{(r)}$ has type $(k_r,\ell_r)$.

To prove the second inequality in \eqref{converse:prop:eq:claim}, we apply
Lemma \ref{constr:lemma1bis} (iii).  For $i=r+1,\dots,t$, it gives
\[
 \dist_\pw(V^{(i-1)},V^{(i)})^{d_\pw/d}
  \le e^{4}\frac{H(\langle\ux^{(i)}_1,\dots,\ux^{(i)}_m\rangle_K)}{H(\ux^{(i)}_1)\cdots H(\ux^{(i)}_m)}
  \le e^4\frac{c_5}{C^{d_\pw/d}}
       \frac{H(\langle\ux^{(i-1)}_1,\dots,\ux^{(i-1)}_m\rangle_K)}{H(\ux^{(i-1)}_1)\cdots H(\ux^{(i-1)}_m)},
\]
and so, for the same values of $i$, we obtain
\[
 \dist_\pw(V^{(i-1)},V^{(i)})^{d_\pw/d}
  \le e^{4}\left(\frac{c_5}{C^{d_\pw/d}}\right)^{i-r}
       \frac{H(\langle\ux^{(r)}_1,\dots,\ux^{(r)}_m\rangle_K)}{H(\ux^{(r)}_1)\cdots H(\ux^{(r)}_m)}
  \le e^{4}\left(\frac{c_5}{C^{d_\pw/d}}\right)^{i-r}.
\]
Since $C\ge 2^n(e^4c_5)^{d/d_\pw}$, this yields $\dist_\pw(V^{(i-1)},V^{(i)})\le
2^{-(i-r)n}$ for $i=r+1,\dots,t$ and so by the triangle inequality of
Lemma \ref{metric:lemma:triangle} we obtain
\[
 \dist_\pw(V^{(r)},V^{(t)})
  \le \sum_{i=r+1}^t \dist_\pw(V^{(i-1)},V^{(i)})
  \le \frac{1}{2^{n-1}} \le \frac{1}{2^{m}},
\]
which completes the proof of \eqref{converse:prop:eq:claim}.

By the above, the sequence $\wx^{(i)}$ is almost orthogonal in $K_\pw^n$
for each $i$ with $-1\le i<s$.  For those $i$, take $V^{(i)}$ to be
the subspace of $K_\pw^n$ spanned by $\wx^{(i)}$ so that
$(V^{(i)})^\perp=\langle\uu_i\rangle_{K_\pw}$.  When $0\le i<s$, both
$\wx^{(i-1)}$ and $\wx^{(i)}$ are almost orthogonal subsequences of
$\ux^{(i)}$.  Then Lemmas \ref{metric:lemma:distdual} and \ref{constr:lemmaV}
yield
\[
 \dist_\pw(\uu_{i-1},\uu_i)^{d_\pw/d}
 = \dist_\pw(V^{(i-1)},V^{(i)})^{d_\pw/d}
 \le \frac{e^4}{H(\ux_1^{(i)})\cdots H(\ux_n^{(i)})}
 \quad (0\le i<s)
\]
because $\langle\ux_1^{(i)},\dots,\ux_n^{(i)}\rangle_K=K^n$ has height $1$.
Since the basis $\ux^{(i)}$ is admissible of size $a^{(i)}$,
Lemma \ref{constr:lemma0} further gives
\[
 \log H(\ux_j^{(i)}) \ge a_j^{(i)}(d_\pw/d)\log(C) = ca_j^{(i)}
 \quad
\text{for $j=1,\dots,n$.}
\]
Using  \eqref{converse:eq:qi}, we conclude that
\[
 \dist_\pw(\uu_{i-1},\uu_i)^{d_\pw/d}
 \le \exp\Big(4-\sum_{j=1}^n ca_j^{(i)}\Big) = \exp(4-q_i)
 \quad (0\le i<s).
\]
Since $(q_i)_{0\le i<s}$ is a strictly increasing sequence of multiples of $c$
and since $cd/d_\pw = \log(C)\ge \log(2)$, we deduce that
\[
 \dist_\pw(\uu_{j-1},\uu_j)
   \le \exp((4-q_j)d/d_\pw) \le (1/2)^{j-i}\exp((4-q_i)d/d_\pw)
\]
for each pair of integers $0\le i<j<s$.  Then, \eqref{converse:prop:eq0}
follows from the triangle inequality of Lemma \ref{metric:lemma:triangle}.
\end{proof}

For the corollary below, we recall our convention that $q_s=\infty$ when $s<\infty$.
As a special case of \eqref{points:eq:LxiX}, we also recall that, for each
$\ux\in K^n$, each non-zero $\uxi\in K_\pw^n$ and each $q\ge 0$, we have
by definition
\[
 L_\uxi(\ux,q)=\max\{\log H(\ux),q+\log D_\uxi(\ux)\}.
\]

\begin{cor}
\label{converse:cor1}
Under the hypotheses of Proposition \ref{converse:prop}, there is a unit vector
$\uxi\in K_\pw^n$ such that, for each integer $i$ with $0\le i<s$ and each
$q\in[q_i,q_{i+1})$, we have
\[
 L_\uxi(\ux_j^{(i)},q)
  \le c_6 + R_j(q_i)
      + \begin{cases}
          q-q_i &\text{if $j=k_i$,}\\
          0 &\text{otherwise,}
          \end{cases}
\]
where $c_6=6+\log(c_5)$.
\end{cor}

\begin{proof}
Consider the sequence of unit vectors $(\uu_i)_{-1\le i<s}$ given by
the proposition.  If $s=\infty$, it follows from \eqref{converse:prop:eq0}
that its image in $\bP^{n-1}(K_\pw)$
converges to the class of a unit vector $\uxi\in K_\pw^n$ such that
\begin{equation}
\label{converse:cor1:eq1}
 \dist_\pw(\uu_i,\uxi)\le 2^\delta\exp((4-q_{i+1})d/d_\pw)
 \quad (-1\le i<s).
\end{equation}
When $s<\infty$, we have $q_s=\infty$ by convention, and the same holds
with $\uxi=\uu_{s-1}$.

We now fix $i$ and $j$ with $0\le i<s$ and $1\le j\le n$.  For each $q\ge 0$, we have
\begin{equation}
\label{converse:cor1:eq2}
 L_\uxi(\ux_j^{(i)},q)
  = \log H(\ux_j^{(i)})
   + \max\left\{0,\
      q + \frac{d_\pw}{d}
            \log\frac{|\ux_j^{(i)}\cdot\uxi|_\pw}{\norm{\ux_j^{(i)}}_\pw}\right\}.
\end{equation}
We also have $\ux_{k_i}^{(i)}\cdot\uu_{i-1}=0$ because $\ux_{k_i}^{(i)}$ belongs to
the sequence $\wx^{(i-1)}$.  So \eqref{metric:eq:dot} yields
\[
 |\ux_{k_i}^{(i)}\cdot\uxi|_\pw
 \le 2^\delta \norm{\ux_{k_i}^{(i)}}_\pw\dist_\pw(\uu_{i-1},\uxi).
\]
If $j\neq k_i$,  we have instead $\ux_j^{(i)}\cdot\uu_i=0$ because $\ux_j^{(i)}$
belongs to $\wx^{(i)}$, and so \eqref{metric:eq:dot} yields
\[
 |\ux_j^{(i)}\cdot\uxi|_\pw
 \le 2^\delta \norm{\ux_j^{(i)}}_\pw\dist_\pw(\uu_i,\uxi).
\]
Using \eqref{converse:cor1:eq1} and noting that $4^\delta\le \exp(2d/d_\pw)$,
we deduce that
\[
  \frac{d_\pw}{d}
            \log\frac{|\ux_j^{(i)}\cdot\uxi|_\pw}{\norm{\ux_j^{(i)}}_\pw}
  \le \begin{cases}
         6-q_i &\text{if $j=k_i$,}\\
         6-q_{i+1}  &\text{otherwise.}
       \end{cases}
\]
Since $\ux^{(i)}$ has size $\ua$, Lemma \ref{constr:lemma0} gives
$\log H(\ux_j^{(i)})\le \log(c_5) + ca_j^{(i)}=c_6 - 6 + R_j(q_i)$.
The conclusion follows by substituting the last two estimates into
\eqref{converse:cor1:eq2}.
\end{proof}

We can now complete the proof of Theorem A as follows.

\begin{cor}
\label{converse:cor2}
Let $\uxi$ be as in Corollary \ref{converse:cor1}.  There is a constant
$c_7=c_7(K,\pw,n)$ such that
\begin{equation}
 \label{converse:cor2:eq}
 \max_{1\le j\le n}|L_{\uxi,j}(q)-R_j(q)|\le c_7
\end{equation}
for each $q\ge 0$.
\end{cor}

\begin{proof}
Fix an integer $i$ with $0\le i<s$ and a point $q\in[q_i,q_{i+1})$.  Denote by
\[
 (r_1,\dots,r_n)
  = \Phi\big(L_\uxi(\ux_1^{(i)},q),\dots,L_\uxi(\ux_n^{(i)},q)\big)
\]
the numbers $L_\uxi(\ux_j^{(i)},q)$ with $1\le j\le n$ written in
monotonically increasing order.  Since $\ux^{(i)}$ is a basis of $\cO_S^n$ over
$\cO_S$, it is also a basis of $K^n$ over $K$, and so by definition we have
\[
 L_{\uxi,j}(q) \le r_j \quad (1\le j\le n).
\]
By Corollary \ref{converse:cor1} and formula \eqref{converse:eq:R}, we also
have
\[
 r_j \le c_6+R_j(q) \quad (1\le j\le n).
\]
By definition of an $n$-system, we further have $\sum_{j=1}^n R_j(q)=q$.  Thus,
by Lemma \ref{points:lemma2} (for $k=1$) and the estimates of \eqref{points:prop1:eq},
we find
\[
 \sum_{j=1}^n \big(c_6+R_j(q)-L_{\uxi,j}(q)\big)
  =nc_6 + q - \sum_{j=1}^n L_{\uxi,j}(q) \le c_6+c_7
\]
for a constant $c_7=c_7(K,\pw,n)\ge \max\{c_6,q_0\}$.  This yields
\eqref{converse:cor2:eq} because each summand in the main
sum on the left is non-negative.  Finally, \eqref{converse:cor2:eq} also holds
for each $q\in[0,q_0)$, because for such $q$, the numbers $R_j(q)$ and $L_{\uxi,j}(q)$
belong to $[0,q_0)\subseteq [0,c_7)$.
\end{proof}

%
%

\section{Spectra of exponents of approximation}
\label{sec:spectra}

We fix a place $\pw$ of $K$ and an integer $n\ge 2$.   For each non-zero
$\uxi\in K_\pw^n$, we write $\omegahat(\uxi)$ for $\omegahat(\uxi,K,\pw)$ and
similarly for the three other exponents introduced in section \ref{results:ssec:exp}.
Our goal is to show that their spectrum is independent of the choice of $K$ and $\pw$
and more precisely that it can be expressed in terms of $n$-systems, as mentioned in the
introduction.   We will also generalize to the present setting the exponents of Laurent
from \cite{La2009b} and show that the same applies to their spectrum.   We start
with the following observation.

\begin{lemma}
\label{spectra:lemma:limL/q}
For each non-zero $\uxi\in K_\pw^n$, we have
\begin{align*}
 &\liminf_{q\to\infty}\frac{L_{\uxi,1}(q)}{q}=\frac{1}{\omega(\uxi)+1},
 &&\limsup_{q\to\infty}\frac{L_{\uxi,1}(q)}{q}=\frac{1}{\omegahat(\uxi)+1},\\
 &\liminf_{q\to\infty}\frac{L^*_{\uxi,1}(q)}{q}=\frac{1}{\lambda(\uxi)+1},
 &&\limsup_{q\to\infty}\frac{L^*_{\uxi,1}(q)}{q}=\frac{1}{\lambdahat(\uxi)+1}.
\end{align*}
\end{lemma}

\begin{proof}
By Definition \ref{results:def:exp}, the number $\omegahat(\uxi)$
(resp.~$\omega(\uxi)$) is the supremum of all $\omega\ge 0$ such that, for each
sufficiently large $t>0$ (resp.~for arbitrarily large $t>0$), there is a non-zero
point $\ux\in K^n$ with $H(\ux)\le e^t$ and $D_\uxi(\ux)\le e^{-\omega t}$.
By definition of $L_{\uxi,1}$ in section \ref{results:ssec:L}, the existence of such
$\ux$ translates into $L_{\uxi,1}((\omega+1)t)\le t$.   Using the change of variables
$q=(\omega+1)t$, we deduce that $\omegahat(\uxi)$ (resp.~$\omega(\uxi)$)
is the supremum of all $\omega\ge 0$ for which $q^{-1}L_{\uxi,1}(q)\le 1/(\omega+1)$
for each sufficiently large $q>0$ (resp.~for arbitrarily large $q>0$).  This yields
the first row of formulas.  The second one is proved in the same way.
\end{proof}

For any function $\uP=(P_1,\dots,P_n)\colon[0,\infty)\to\bR^n$ and any $j=1,\dots,n$,
we define
\[
 \phibot_j(\uP)=\liminf_{q\to\infty}\frac{P_j(q)}{q}
 \et
 \phitop_j(\uP)=\limsup_{q\to\infty}\frac{P_j(q)}{q}.
\]
The following generalization of \cite[Corollary 1.4]{R2015} characterizes the spectrum
of $(\omega,\omegahat,\lambda,\lambdahat)$.

\begin{proposition}
\label{spectra:prop:4exp}
The set $\cS$ of quadruples
\[
 \left(
     \frac{1}{\omega(\uxi)+1}, \frac{1}{\omegahat(\uxi)+1}, \frac{1}{\lambda(\uxi)+1},
     \frac{1}{\lambdahat(\uxi)+1}
 \right) \in [0,1]^4
\]
where $\uxi\in K_\pw^n$ has $K$-linearly independent coordinates coincides with the
set of quadruples
\[
 \big(\phibot_1(\uP), \phitop_1(\uP), 1-\phitop_n(\uP), 1-\phibot_n(\uP)\big) \in [0,1]^4
\]
where $\uP=(P_1,\dots,P_n)$ is an $n$-system with first component $P_1$ unbounded.
\end{proposition}

\begin{proof}
By the lemma, $\cS$ consists of the points $(\phibot_1(\uL_\uxi), \phitop_1(\uL_\uxi),
\phibot_1(\uL^*_\uxi), \phitop_1(\uL^*_\uxi))$ for all $\uxi\in K_\pw^n$ with
$K$-linearly independent coordinates.  By definition of $L_{\uxi,1}$ in section \ref{results:ssec:L},
that condition on $\uxi$ is equivalent to asking that $L_{\uxi,1}$ is unbounded.  Thus,
by Theorem A, the set $\cS$ consists of the points
$(\phibot_1(\uP), \phitop_1(\uP), \phibot_1(\uP^*), \phitop_1(\uP^*)\big)$ where
$\uP=(P_1,\dots,P_n)$ is an $n$-system whose first component $P_1$
is unbounded.  The conclusion follows since, for any $n$-system $\uP$, one has
$\phibot_1(\uP^*)=1-\phitop_n(\uP)$ and $\phitop_1(\uP^*)=1-\phibot_n(\uP)$ .
\end{proof}

\begin{cor}
\label{spectra:cor}
The set $\cS$ is independent of the choice of $K$ and $\pw$.  In particular,
since Jarn\'{\i}k's identity \eqref{intro:eq:Jarnik} holds for any point $\uxi$ of
$\bQ_\infty^3=\bR^3$ with $\bQ$-linearly independent coordinates, it also holds
for any point $\uxi$ of $K_\pw^3$ with $K$-linearly independent coordinates.
\end{cor}

To generalize the exponents of Laurent from \cite{La2009b}, we fix an integer $k$
with $1\le k\le n-1$, and we assume for simplicity that $\norm{\uxi}_\pw=1$.
For each non-zero $\uX\in\tbigwedge^kK^n$, we define
\[
 D_\uxi^*(\uX)=\norm{\uxi\wedge \uX}_\pw^{d_\pw/d}
    \prod_{\pv\neq\pw}\norm{\uX}_\pv^{d_\pv/d}
\]
in addition to the quantity $D_\uxi(\uX)$ from Definition \ref{points:def:L}.
For $k=1$, this agrees with the notation of section \ref{results:ssec:exp}
since $\norm{\uxi}_\pw=1$.   We note the following fact.

\begin{lemma}
\label{spectra:lemma:equiv}
For given $\omega\ge 0$ and $Q\ge 1$, the following conditions are equivalent:
\begin{itemize}
 \item[{\rm (i)}] there exists a non-zero $\uX\in\tbigwedge^kK^n$ with $H(\uX)\le Q$ and
 $D_\uxi^*(\uX)\le Q^{-\omega}$;
 \smallskip
 \item[{\rm (ii)}] there exists a non-zero $\uY\in\tbigwedge^{n-k}K^n$ with $H(\uY)\le Q$ and
 $D_\uxi(\uY)\le Q^{-\omega}$.
\end{itemize}
\end{lemma}

\begin{proof}
Consider the $K_\pv$-linear isometry $\varphi_{k,\pv}\colon\tbigwedge^kK_\pv^n \to
\tbigwedge^{n-k}K_\pv^n$ from section \ref{metric:ssec:duality} for  each
$\pv\in\MK$.  They all restrict to a single $K$-linear isomorphism
$\varphi_k\colon\tbigwedge^kK^n \to \tbigwedge^{n-k}K^n$.  Moreover, for
each $\uX\in\tbigwedge^kK_\pw^n$, we have
\[
 \uxi\iprod\varphi_{k,\pw}(\uX) = \varphi_{k+1,\pw}(\uX\wedge\uxi)
\]
(cf.\ \cite[\S3, Lemma 2]{BL2010}) and thus $\norm{\uxi\iprod\varphi_{k,\pw}(\uX)}_\pw
=\norm{\uX\wedge\uxi}_\pw$.  We conclude that, if $\uX\in \tbigwedge^kK^n$
is non-zero, then the point $\uY=\varphi_k(\uX)\in \tbigwedge^{n-k}K^n$
is non-zero with $H(\uY)=H(\uX)$ and $D_\uxi(\uY)=D^*_\uxi(\uX)$.   Thus
the two conditions are equivalent.
\end{proof}

\begin{definition}
\label{spectra:def:expL}
We denote by $\omega_{k-1}(\uxi)$ (resp.\ $\omegahat_{k-1}(\uxi)$) the supremum of all
$\omega\ge 0$ for which the equivalent conditions of Lemma \ref{spectra:lemma:equiv}
are fulfilled for arbitrarily large values of $Q$ (resp.\ for all sufficiently large values of $Q$).
\end{definition}

For $k=1$, applying condition (i) shows that $\omega_0(\uxi)=\lambda(\uxi)$ and
$\omegahat_0(\uxi)=\lambdahat(\uxi)$.  For $k=n-1$, applying condition (ii) instead shows
that $\omega_{n-2}(\uxi)=\omega(\uxi)$ and $\omegahat_{n-2}(\uxi)=\omegahat(\uxi)$.
Moreover, for $K=\bQ$, $\pw=\infty$ and any choice of $k$, the number $\omega_{k-1}(\uxi)$
defined above coincides with the exponent of Laurent from \cite{La2009b},
in view of condition (i): see \cite[\S2, Remark]{La2009b} or \cite[\S4, Proposition]{BL2010}.

Arguing as in Lemma \ref{spectra:lemma:limL/q}, using condition (ii) and the definition
of $L_{\uxi,1}^{(k-1)}$, we deduce that
\[
\frac{1}{\omega_{k-1}(\uxi)+1}=\liminf_{q\to\infty} \frac{L_{\uxi,1}^{(n-k)}(q)}{q}
\et
\frac{1}{\omegahat_{k-1}(\uxi)+1}=\limsup_{q\to\infty} \frac{L_{\uxi,1}^{(n-k)}(q)}{q}.
\]
On the other hand, the proof of Proposition \ref{points:prop1} shows that $L_{\uxi,1}^{(j)}$
differs by a bounded function from $L_{\uxi,1}+\cdots+L_{\uxi,j}$ for each $j=1,\dots,n-1$.
We conclude that the spectrum of these $2n-2$ exponents is independent of $K$
and $\pw$, and characterized as follows (cf.\ \cite[Proposition 3.1]{R2016}).

\begin{proposition}
\label{spectra:prop:allexp}
The set of points
\[
 \left(
     (\omega_0(\uxi)+1)^{-1},\dots,(\omega_{n-2}(\uxi)+1)^{-1},
     (\omegahat_0(\uxi)+1)^{-1}, \dots, (\omegahat_{n-2}(\uxi)+1)^{-1}
 \right)
\]
where $\uxi\in K_\pw^n$ has $K$-linearly independent coordinates
coincides with the set of points
\[
 \big(\psibot_{n-1}(\uP), \dots, \psibot_{1}(\uP), \psitop_{n-1}(\uP), \dots, \psitop_1(\uP)\big)
\]
where $\uP=(P_1,\dots,P_n)$ is an $n$-system with first component $P_1$ unbounded, and
\[
 \psibot_j(\uP)=\liminf_{q\to\infty}\frac{P_1(q)+\cdots+P_j(q)}{q}
 \et
 \psitop_j(\uP)=\limsup_{q\to\infty}\frac{P_1(q)+\cdots+P_j(q)}{q}
\]
 for $j=1,\dots,n-1$.
\end{proposition}

As in \cite[Definition 2]{La2009b}, the exponent $\omega_{k-1}(\uxi)$ can be
described geometrically as the supremum of all $\omega\ge 0$ for which there are
infinitely many subspaces $V$ of $K^n$ of dimension $k$ that satisfy
\[
 \inf\{\dist_\pw(\uxi,\ux)\,;\,\ux\in V\} \le H(V)^{-(\omega+1)d/d_\pw}.
\]

We simply sketch the proof which is similar in spirit to that of \cite[\S4, Proposition]{BL2010},
based on the results of section \ref{sec:points}.  We first observe that, in defining
$\omega_{k-1}(\uxi)$ through condition (ii) of Lemma \ref{spectra:lemma:equiv}, we may
take for $\uY$ a point of $\tbigwedge^{n-k}K^n$ which realizes the first 
minimum of $\cC^{(n-k)}_\uxi(t)$
where $t=(\omega+1)\log(Q)$.  Since that convex body is comparable to
$\tbigwedge^{n-k}\cC_\uxi(t)$, we may even take for $\uY$ the wedge product of
the first $n-k$ points of a basis of $K^n$ which realizes the minima of $\cC_\uxi(t)$
(see the comments after Theorem \ref{adelic:thm:Burger}).  Let $W$ denote the subspace
of $K^n$ spanned by these points, so that $\uY$ spans $\tbigwedge^{n-k}W$, and
let $V=W^\perp$.  Then, going back to the proof of the lemma,  we find that
$\varphi_k(\tbigwedge^kV)=\tbigwedge^{n-k}W$ and so $\uY=\varphi_k(\uX)$
for some generator $\uX$ of $\tbigwedge^kV$.  In defining $\omega_{k-1}(\uxi)$
through condition (i), we may thus assume that $\uX$ has this form.  For such a point,
Lemma \ref{metric:lemmaCbis} yields $D_\uxi^*(\uX) = \dist_\pw(\uxi,\overline{V})^{d_\pw/d}H(V)$
where $\overline{V}=\langle V\rangle_{K_\pw}$ is the topological closure of $V$ in
$K_\pw^n$, and the claim follows.

%
%

\section{The principle of Thunder}
\label{sec:principle}

In his alternative proof of the adelic Minkowski's theorem
\cite{Th2002}, Jeff Thunder relates the successive minima
of a convex body of $K_\bA^n$ to those of an appropriate
convex body of $\bQ_\bA^{dn}$.  We formulate his idea below
as a general principle.

To this end, we use the following construction where $\bQ^{dn}$
and $K^n$ are viewed respectively as subsets of $\bQ_\bA^{dn}$
and $K_\bA^n$ under the diagonal embedding.

\begin{lemma}
\label{lemma:T}
Let $T\colon \bQ^{dn}\to K^n$ be a $\bQ$-linear isomorphism.  For
each place $\pu$ of $\bQ$ and each place $\pv$ of $K$ above $\pu$, we
denote by $T_\pv\colon\bQ_\pu^{dn}\to K_\pv$ the $\bQ_\pu$-linear map
which extends $T$.  Then
\[
 \begin{array}{rrcl}
   T_\pu\,: &\bQ_\pu^{dn}&\longrightarrow &\prod_{\pv\mid\pu}K_\pv^n \\
   &\ux\ &\longmapsto &(T_\pv(\ux))_{\pv\mid\pu}
 \end{array}
\]
is a $\bQ_\pu$-linear isomorphism.  Moreover the map
\[
 \begin{array}{rrcl}
   T_\bA\,: &\bQ_\bA^{dn}&\longrightarrow &K_\bA^n \\
   &(\ux_\pu)_{\pu\in M(\bQ)}&\longmapsto &\big((T_\pv(\ux_\pu))_{\pv\mid\pu}\big)_{\pu\in M(\bQ)}
 \end{array}
\]
is the unique $\bQ_\bA$-linear map which extends $T$.  It is also an isomorphism.
\end{lemma}

\begin{proof}
By construction the map $T_\pu$ is $\bQ_\pu$-linear with domain and
codomain of the same dimension $dn=\sum_{\pv\mid\pu}d_\pv n$ as vector
spaces over $\bQ_\pu$.  Moreover $T_\pu(\bQ^{dn})$ is the image
of $K^n$ in $\prod_{\pv\mid\pu}K_\pv^n$ under the diagonal embedding,
which is dense in this product.  So $T_\pu$ is surjective and therefore
it is an isomorphism.  This proves the first assertion, and the
others follow from it.
\end{proof}

\begin{proposition}[Thunder's principle]
\label{prop:Thunder_principle}
With the notation of Lemma \ref{lemma:T}, let $\cK$ be a convex
body of $K_\bA^n$.  Then, $\cC:=T_\bA^{-1}(\cK)$ is a convex body of
$\bQ_\bA^{dn}$ and we have
\begin{equation*}
 \label{prop:Thunder_principle:eq1}
 \lambda_{d(i-1)+j}(\cC)\asymp \lambda_i(\cK)
 \quad
 (1\le i\le n,\ 1\le j\le d),
\end{equation*}
with implied constants that depend only on $K$, $n$ and $T$.
\end{proposition}

\begin{proof}
Writing $\cK=\prod_{\pv\in \MK}\cK_\pv$, we find that
$\cC=\prod_{u\in M(\bQ)}\cC_u$ where
\[
 \cC_\pu=\bigcap_{\pv\mid\pu}T_\pv^{-1}(\cK_\pv)
\]
is a convex body of $\bQ_\pu^{dn}$ for each $\pu\in M(\bQ)$.  For all but
finitely many places $\pu\neq \infty$, we also have that $T_\pu(\bZ_\pu^{dn})=
\prod_{\pv\mid\pu}\cOv^n=\prod_{\pv\mid\pu}\cK_\pv$ and so
$\cC_\pu=\bZ_\pu^{dn}$.  Thus $\cC$ is a convex body of $\bQ_\bA^{dn}$.

Choose linearly independent elements $\uy_1,\dots,\uy_n$ of $K^n$
over $K$ which realize the successive minima of $\cK$, choose
a basis $\uomega=(\omega_1,\dots,\omega_d)$ of the ring of integers
$\cO_K$ of $K$ as a $\bZ$-module, and let
$c=\max\{|\uomega_j|_\pv\,;\,1\le j\le d \text{ and }\pv\mid\infty\}$.  Then the points
\begin{equation*}
 \label{prop:Thunder_principle:eq2}
 \ux_{i,j}:=T^{-1}(\omega_j\uy_i) \in \bQ^{dn}
 \quad
 (1\le i\le n,\ 1\le j\le d)
\end{equation*}
are linearly independent over $\bQ$.  For each indexing pair $(i,j)$, we
also have $\uy_i\in\lambda_i(\cK)\cK$ and $\omega_j\cK\subseteq c\cK$,
thus $\omega_j\uy_i\in c\lambda_i(\cK)\cK$.  As $T_\infty$ is
linear over $\bQ_\infty=\bR$, this means that
\[
  \ux_{i,j} \in T_\bA^{-1}(c\lambda_i(\cK)\cK)=c\lambda_i(\cK)\cC.
\]
In view of the linear independence of the points $\ux_{i,j}$ and the fact that
$\lambda_1(\cK)\le \cdots\le \lambda_n(\cK)$, we conclude that
\[
 \lambda_{d(i-1)+j}(\cC)\le c\lambda_i(\cK)
 \quad
 (1\le i\le n,\ 1\le j\le d).
\]
On the other hand, since $T_\bA$ is $\bQ_\bA$-linear and invertible,
we have $\mu(\cC)=c'\mu(\cK)$ for some constant $c'>0$ which depends only
on $T_\bA$. So, the adelic Minkowski's theorem \ref{adelic:thm:MBV},
applied to $\cC$ and $\cK$ separately, yields
\[
 \lambda_1(\cC)\cdots\lambda_{dn}(\cC)
   \asymp \mu(\cC)^{-1}
   \asymp \mu(\cK)^{-1}
   \asymp (\lambda_1(\cK)\cdots\lambda_n(\cK))^d.
\]
The conclusion follows.
\end{proof}

%
%

\section{Proofs of Theorems B and C}
\label{sec:proofBC}

We fix a basis $\ualpha=(\alpha_1,\dots,\alpha_d)$ of $K$ over $\bQ$,
a place $\pw$ of $K$ of local degree $d_\pw=1$ above a place
$\ell$ of $\bQ$, so that $K_\pw=\bQ_\ell$, and a non-zero
point $\uxi=(\xi_1,\dots,\xi_n)\in K_\pw^n=\bQ_\ell^n$,
assuming $n\ge 2$.  We form
\[
 \Xi=\ualpha\otimes\uxi=\big(\alpha_1\uxi,\dots,\alpha_d\uxi\big)
     \in (K_\pw^n)^d=(\bQ_\ell^n)^d,
\]
and note that
\begin{equation}
 \label{proofBC:eq:normXi}
 \norm{\Xi}_\ell = \norm{\ualpha}_\pw \norm{\uxi}_\pw.
\end{equation}
In order to apply the results of section \ref{sec:points}, we need to adjust Definition
\ref{points:def:C} so that the family of convex bodies in $\bQ_\bA^{dn}$ attached to $\Xi$
and the family of convex bodies in $K_\bA^n$ attached to $\uxi$ do not depend
on the norms of those points.  Thus, for each $q\ge 0$, we define
\begin{align*}
 \cCXi(q)
  &=\big\{\, (\ux_u)\in \bQ_\bA^{dn} \ ;\
      \norm{\Xi}_\ell^{-1} |\ux_\ell\cdot\Xi|_\ell\le e^{-q}
      \,\, \ \text{and}\
      \norm{\ux_\pu}_\pu\le 1 \ \text{for each $\pu\in M(\bQ)$}\,
      \big\}\,,\\
 \cCxi(q)
  &=\big\{\, (\uy_\pv)\in K_\bA^{n} \,\ ;\
      \norm{\uxi}_\pw^{-1}|\uy_\pw\cdot\uxi|_\pw\le e^{-qd}
      \, \ \text{and}\
      \norm{\uy_\pv}_\pv\le 1 \ \text{for each $\pv \in \MK$}
      \big\}\,,
\end{align*}
since $d_\pw=1$.
In order to relate the minima of these convex bodies and to deduce relationships
between the standard four exponents of approximation attached to the triples
$(\Xi,\bQ,\ell)$ and $(\uxi,K,\pw)$, we consider the $\bQ$-linear isomorphism
$T\colon (\bQ^n)^d\to K^n$ given by
\begin{equation*}
\label{eq:T}
 T(\ux_1,\dots,\ux_d) = \alpha_1\ux_1+\cdots+\alpha_d\ux_d
\end{equation*}
for any $\ux_1,\dots,\ux_d\in\bQ^n$.  For each place $\pu$ of $\bQ$
and each place $\pv$ of $K$ above $\pu$, it extends by continuity to
a $\bQ_\pu$-linear map $T_\pv$ from $(\bQ_\pu^n)^{d}$ to $K_\pv^n$
given by
\begin{equation*}
\label{eq:Tv}
 T_\pv(\ux_1,\dots,\ux_d) = \alpha_1\ux_1+\cdots+\alpha_d\ux_d
\end{equation*}
for any $\ux_1,\dots,\ux_d\in\bQ_\pu^n$.  Following Lemma \ref{lemma:T},
this yields a $\bQ_\pu$-linear isomorphism $T_\pu$ from
$\bQ_\pu^{nd}$ to $\prod_{\pv\mid\pu}K_\pv^n$, as well as
a $\bQ_\bA$-linear isomorphism $T_\bA$ from
$\bQ_\bA^{nd}$ to $K_\bA^n$.

\begin{proposition}
\label{proofBC:prop}
There is an id\`ele $\ua\in K^\mult_\bA$ which depends only on $K$, $\pw$,
$\ualpha$ and $n$ such that, for each $q\ge 0$,
\[
 \ua^{-1}\cCXi(dq) \subseteq T_\bA^{-1}(\cCxi(q)) \subseteq \ua\,\cCXi(dq).
\]
\end{proposition}

\begin{proof}
Fix a place $\pu$ of $\bQ$.  Since $T_\pu$ is a $\bQ_\pu$-linear isomorphism,
there exists a constant $c_u\ge 1$ such that
\[
 c_u^{-1}\norm{\ux}_\pu \le \max_{\pv\mid\pu}\norm{T_\pv(\ux)}_\pv \le c_\pu\norm{\ux}_\pu
\]
for all $\ux\in\bQ_\pu^{nd}$.  Since $T_\pu(\bZ_\pu^{nd})=\prod_{\pv\mid\pu}\cOv^n$
for all but finitely many $\pu\neq\infty$, we may take $c_\pu=1$ for those $\pu$.
When $\pu=\ell$ and $\ux=(\ux_1,\dots,\ux_d)\in(\bQ_\ell^n)^d$, we also find
\[
\ux\cdot\Xi=\sum_{i=1}^d\alpha_i\ux_i\cdot\uxi=T_\pw(\ux)\cdot\uxi,
\]
thus $|\ux\cdot\Xi|_\ell=|T_\pw(\ux)\cdot\uxi|_\pw$.  Then, using \eqref{proofBC:eq:normXi}
and assuming that $c_\ell\ge \max\{\norm{\ualpha}_\pw,\norm{\ualpha}_\pw^{-1}\}$,
we deduce that
\[
 c_\ell^{-1} \frac{|\ux\cdot\Xi|_\ell}{\norm{\Xi}_\ell}
   \le \frac{|T_\pw(\ux)\cdot\uxi|_\pw}{\norm{\uxi}_\pw}
   \le c_\ell \frac{|\ux\cdot\Xi|_\ell}{\norm{\Xi}_\ell}.
\]

Choose an id\`ele $\ua=(a_\pu) \in\bQ_\bA^\mult$ such that $|a_\pu|_\pu\ge c_\pu$
for each $\pu\in M(\bQ)$.  Then, if $\ux=(\ux_\pu) \in\bQ_\bA^{nd}$ and
$\uy=(\uy_\pv) \in K_\bA^n$ are related by $\uy=T_\bA(\ux)$, the above
estimates yield
\[
 \norm{a_\pu^{-1}\ux_u}_\pu
   \le \max_{\pv\mid\pu}\norm{\uy_\pv}_\pv
   \le \norm{a_\pu\ux_\pu}_\pu
 \et
 \frac{|a_\ell^{-1}\ux_\ell\cdot\Xi|_\ell}{\norm{\Xi}_\ell}
   \le \frac{|\uy_\pw\cdot\uxi|_\pw}{\norm{\uxi}_\pw}
   \le \frac{|a_\ell\ux_\ell\cdot\Xi|_\ell}{\norm{\Xi}_\ell}.
\]
So, if $\uy\in\cCxi(q)$ (resp.\ $\ua\ux\in\cCXi(dq)$) for some $q\ge 0$,
then $\ua^{-1}\ux\in\cCXi(dq)$ (resp.\ $\uy\in\cCxi(q)$).  This means that
$\ua^{-1}T_\bA^{-1}(\cCxi(q)) \subseteq \cCXi(dq)$
(resp.\  $\ua^{-1}\cCXi(dq) \subseteq T_\bA^{-1}(\cCxi(q))$), as needed.
\end{proof}

By Proposition \ref{dilations:prop} and Thunder's principle
(Proposition \ref{prop:Thunder_principle}), the above result yields
the following estimates.

\begin{cor}
\label{proofBC:cor}
For each $i=1,\dots,n$, each $j=1,\dots,d$ and each $q\ge 0$, we have
\begin{equation}
\label{proofBC:cor:eq}
 \lambda_{d(i-1)+j}(\cCXi(dq))
   \asymp \lambda_{i}(\cCxi(q))
\end{equation}
with implicit constants that depend only on $K$, $\ualpha$, $n$ and $\pw$.
\end{cor}

\begin{proof}[\textbf{Proof of Theorem B}]
Fix $i$, $j$ and $q$ as above.  Taking logarithms on both sides of \eqref{proofBC:cor:eq}
and using Lemma \ref{points:lemma2}, we obtain that the absolute value of the difference
\begin{equation}
\label{proofBC:proofB:eq}
 L_{\Xi,d(i-1)+j}(dq) - L_{\uxi,i}(q)
\end{equation}
is bounded above by a constant that depends only on $K$, $\ualpha$, $n$ and $\pw$.
Letting $i'=n+1-i$ and $j'=d+1-j$, Lemma \ref{points:lemmaLc} shows that the
same applies to
\[
 L_{\Xi,d(i-1)+j}(dq) + L^*_{\Xi,d(i'-1)+j'}(dq) - dq
 \et
 L_{\uxi,i}(q) + L^*_{\uxi,i'}(q) - q.
\]
Subtracting from the first number the sum of the second and of
\eqref{proofBC:proofB:eq}, we obtain that
\[
 L^*_{\Xi,d(i'-1)+j'}(dq) - L^*_{\uxi,i'}(q) - (d-1)q
\]
also has absolute value bounded above by such a constant.  As $i'$ runs from $1$
to $n$ with $i$, and as $j'$ runs from $1$ to $d$ with $j$, this proves the two
inequalities of Theorem B.
\end{proof}

\begin{cor}
\label{proofBC:cor:exponents}
Upon writing $\omegahat(\uxi)$ for $\omegahat(\uxi,K,\pw)$,
$\omegahat(\Xi)$ for $\omegahat(\Xi,\bQ,\ell)$, and similarly
for the other exponents, we have
\begin{align*}
 &d\big(\omegahat(\uxi)+1\big) = \omegahat(\Xi) +1,
 &&d\big(\omega(\uxi) +1\big) =\omega(\Xi) +1,\\
 &d\Big(\frac{1}{\lambdahat(\uxi)} +1\Big) = \frac{1}{\lambdahat(\Xi)} +1,
 &&d\Big(\frac{1}{\lambda(\uxi)} +1\Big) = \frac{1}{\lambda(\Xi)} +1.
\end{align*}
\end{cor}

\begin{proof}
By Theorem B, the ratios $L_{\Xi,1}(dq)/q$ and $L_{\uxi,1}(q)/q$ have the same
limit points as $q$ goes to infinity. In particular, they have the same superior limit
and the same inferior limit.  Applying Lemma \ref{spectra:lemma:limL/q} separately to
$(\Xi,\bQ,\ell)$ and $(\uxi,K,\pw)$ to compute these limits and comparing
the results, we get the first row of equalities.

By Theorem B, the quantities $L^*_{\Xi,1}(dq)/q - d$ and $L^*_{\uxi,1}(q)/q -1$
also have the same limit points as $q$ goes to infinity and the same lemma yields
the second row of equalities.
\end{proof}

As an application, suppose that $n=3$ and that $\uxi\in K_\pw^3$ has linearly
independent coordinates over $K$.  Then, $\uxi$ satisfies Jarn\'{\i}k's identity
\eqref{intro:eq:Jarnik} by Corollary \ref{spectra:cor}.  Using the formulas of the
above corollary, we deduce the identity \eqref{intro:eq:Jarnik:Xi} relating
$\lambdahat(\Xi)$ and $\omegahat(\Xi)$.

\begin{proof}[\textbf{Proof of Theorem C}]
Let $S$ denote the subset of $K_\pw^3$ from Theorem \ref{results:thm:Bel}
of Bel.  Since the supremum of $\lambdahat(\uxi,K,\pw)$ is $1/\gamma$
as $\uxi$ runs through $S$, the formulas of Corollary \ref{proofBC:cor:exponents}
imply that for the corresponding points $\Xi=\ualpha\otimes\uxi$, the supremum
of $\lambdahat(\Xi,\bQ,\ell)$ is $1/(d\gamma^2-1)$.   Since the points $\uxi$
of $S$ satisfy Jarn\'{\i}k's identity \eqref{intro:eq:Jarnik},  the supremum of
$\omegahat(\uxi,K,\pw)$ is $\gamma^2$ as $\uxi$ runs through $S$.
So,  for the corresponding points $\Xi=\ualpha\otimes\uxi$, we find similarly that
the supremum of $\omegahat(\Xi,\bQ,\ell)$ is $d(\gamma^2+1)-1$.
This proves Theorem C because, for $\uxi=(1,\xi,\xi^2)\in S$, the $3d$ coordinates
of $(\ualpha,\xi\ualpha,\xi^2\ualpha)$ form a permutation of those of
$\Xi=(\alpha_1\uxi,\dots,\alpha_d\uxi)$ and so these two points have the same
exponents of approximation.
\end{proof}

As a final remark, suppose that $\uP=(P_1,\dots,P_n)$ is an $n$-system on
$[0,\infty)$ for which $\uL_\uxi-\uP$ is bounded.  Then the difference
$\uL_\Xi-\uR$ is bounded for the function $\uR=(R_1,\dots,R_{nd})$ from
$[0,\infty)$ to $\bR^{nd}$ given by
\[
 R_{d(i-1)+j}(q)=P_i(q/d) \quad (1\le i\le n, \ 1\le j\le d, \ q\ge 0).
\]
If $d>1$, this is not an $nd$-system because its components are
piecewise linear with slopes $0$ and $1/d$.  However, it is a generalized
$nd$-system in the sense of \cite[Definition 4.5]{R2016} and so it can
easily be approximated uniformly by an $nd$-system as explained in
\cite[section 4]{R2016}.

%
%

\end{document}